\documentclass[twocolumn]{autart}    
\usepackage{cite} 
\usepackage{color}
\usepackage{graphicx}          
\usepackage{amsmath}
\usepackage{subfigure}
\usepackage{amssymb}
\usepackage{amsfonts}
\usepackage{eucal}
\usepackage[titletoc,title]{appendix}
\usepackage{mathptmx} 
\usepackage{epsfig}
\usepackage{epstopdf}

\usepackage{tikz} \usetikzlibrary{calc} \tikzset{>=latex}

\usepackage[colorlinks,linkcolor=blue,urlcolor=blue]{hyperref}

\usepackage{bbm}
\usepackage{mathrsfs} 
\newtheorem{theorem}{Theorem}
\newtheorem{lemm}{Lemma}
\newtheorem{propo}{Proposition}
\newtheorem{remark}{Remark}
\newenvironment{proof}[1][Proof]{\noindent\textbf{#1.} }{\ \rule{0.5em}{0.5em}}

\begin{document}

        \setlength\abovedisplayskip{6pt}
        \setlength\belowdisplayskip{6pt}
        \setlength\abovedisplayshortskip{6pt}
        \setlength\belowdisplayshortskip{6pt}
        \allowdisplaybreaks
        \setlength{\parindent}{1em}
        \setlength{\parskip}{-0em}  
        \addtolength{\oddsidemargin}{6pt}      
        
        \begin{frontmatter}
                
\title{Backstepping Neural Operators for $2\times 2$ Hyperbolic PDEs} 
                
                \thanks[footnoteinfo]{Corresponding author: Mamadou Diagne. 
                }
                
                \author[Wang]{Shanshan Wang}\ead{wss\_dhu@126.com},  
                \author[UMICH]{Mamadou Diagne*}\ead{mdiagne@ucsd.edu},
                \author[UMICH]{Miroslav Krstic}\ead{krstic@ucsd.edu} 
                
                \address[Wang]{Department of Control Science and Engineering, University of Shanghai for Science and Technology, Shanghai, P. R. China, 200093}  
                \address[UMICH]{Department of Mechanical and Aerospace Engineering, University of California San Diego, La Jolla, CA, USA, 92093} 
                
                \begin{keyword}
                    PDE backstepping, deep learning, neural networks, distributed parameter systems
                \end{keyword}
                

\begin{abstract} 
Deep neural network approximation of nonlinear operators, commonly referred to as DeepONet, has proven capable of approximating PDE backstepping designs in which a single Goursat-form PDE governs a single feedback gain function. In boundary control of coupled PDEs, coupled Goursat-form PDEs govern two or more gain kernels---a PDE structure unaddressed thus far with DeepONet. In this paper, we explore the subject of approximating systems of gain kernel PDEs for hyperbolic PDE plants by considering a simple counter-convecting $2\times 2$ coupled system in whose control a $2\times 2$ kernel PDE system in Goursat form arises. Engineering applications include oil drilling, the Saint-Venant model of shallow water waves, and the Aw-Rascle-Zhang model of stop-and-go instability in congested traffic flow. We establish the continuity of the mapping from a total of five plant PDE functional coefficients to the kernel PDE solutions, prove the existence of an arbitrarily close DeepONet approximation to the kernel PDEs, and ensure that the DeepONet-approximated gains guarantee stabilization when replacing the exact backstepping gain kernels. Taking into account anti-collocated boundary actuation and sensing, our $L^2$\emph{-Globally-exponentially} stabilizing (GES) approximate gain kernel-based output feedback design implies the deep learning of both the controller's and the observer's gains. Moreover, the encoding of the output-feedback law into DeepONet ensures \emph{semi-global practical exponential stability (SG-PES).} The DeepONet operator speeds up the computation of the controller gains by multiple orders of magnitude. Its theoretically proven stabilizing capability is demonstrated through simulations.
\end{abstract}

      \end{frontmatter}

\section{Introduction}\label{intro}


The versatility of coupled first-order hyperbolic  PDE systems across various fields extends  to applications such as modeling traffic dynamics \cite{2006Goatin, yu2019traffic}, open channel fluid flow \cite{2003Boundary,diagne2017backstepping,diagne2017control,diagne2012lyapunov}, heat exchanger systems \cite{XU2002}, oil drilling operations \cite{Aamo2016}, flexible pipe assemblies \cite{Ren2006Wave}, natural gas pipeline networks \cite{Gugat2011}, and power transmission line networks \cite{Curr2011A}, among others.
\subsection{Control of coupled linear hyperbolic PDE systems}
The development of stabilizing boundary feedback laws for counter-convective first-order linear hyperbolic systems has evolved since the introduction of the locally exponentially stabilizing boundary controller in \cite{coron1999lyapunov}. This controller was originally crafted for the Saint-Venant model, a nonlinear coupled hyperbolic PDE. Importantly, assuming some compatibility condition of the initial data and restricting the set of admissible equilibrium, \cite{coron1999lyapunov} introduced the simplest entropy-based Lyapunov function to control a system whose total energy is not a suitable Lyapunov candidate. After \cite{coron1999lyapunov}, the Riemann invariants method enabled the determination of gate operating conditions that ensure exponential stability relying exclusively on local water level measurements at the gate positions, in the absence of friction. Adaptive backstepping control of a wave PDE with anti-damping represented in Riemann variables as a coupled first-order hyperbolic PDE was developed in \cite{bresch2014output}. Fundamental contributions, \cite{bastin2010further, bastin2011boundary}, along with \cite{Vazquez2011}, have led to the formulation of a quadratic Lyapunov candidate for $2\times 2$ linear hyperbolic systems. Our work revolves around the PDE backstepping approach, which relies on a single boundary actuation and a full-state measurement, as outlined in the reference \cite{Vazquez2011}.  The approach described in \cite{bastin2010further, bastin2011boundary} could be referred to as the ``dissipativity" method as it pertains to finding dissipative boundary conditions that are similar to ``small gain conditions." Furthermore, results in \cite{bastin2010further, bastin2011boundary} are based on a dual boundary actuation strategy. They exclusively leverage measurements obtained from the boundary positions, so to speak, at the gate locations of a water canal to achieve stabilization. Results on bilateral or dual boundary control of hyperbolic PDEs by means of PDE backstepping were reported in \cite{vazquez2016bilateral, auriol2017two}. Delay-adaptive boundary control of coupled hyperbolic PDE-ODE cascade systems was recently established in \cite{wang2023delay} via Batch-Least Square Identification (BaLSI) \cite{Karafyllis2019Adaptive}. Alongside these two major approaches, a range of methods, including frequency domain analysis \cite{litrico2009boundary}, differential flatness \cite{rabbani2009feed}, sliding-mode control \cite{tang2014sliding}, and proportional-integral control \cite{dos2012multi}, have been employed for the design of control laws for $2\times 2$  hyperbolic system of balance laws.

Aiming at stabilizing a two-phase slugging phenomenon observed in oil drilling processes \cite{di2012model}, PDE backstepping has been used to control a $2+1$ counter-convective system actuated at one boundary \cite{di2012backstepping}. The problem structure outlined in \cite{di2012model} has broad applicability, appearing in various multiphase flow processes like drift-flux modeling in oil drilling \cite{aarsnes2014control} and coupled water-sediment dynamics in river breaches \cite{diagne2017backstepping}. In the latter case, the backstepping method achieves exponential stabilization of supercritical flow regimes, which was not attainable by the design proposed in \cite{di2012backstepping}. A similar system's structure, but in the $3+1$ form, can be found in \cite{burkhardt2021stop}, where the control of a two-class traffic system is studied.
 
 Generalized results on the exponential stabilization of an arbitrary number of coupled waves among which a single pattern is controlled were achieved in \cite{Meglio2013}. It is a non-trivial extension of the backstepping approach \cite{Vazquez2011,di2012backstepping} to the so-called $n+1$ case where the boundary actuation of a single counter-convective pattern regulates systems with an arbitrary number of coupled waves. Later on, systems consisting of $n+m$ coupled linear hyperbolic PDE systems were exponentially stabilized by actuating $m$ components in \cite{Hu2016} (see \cite{hu2019boundary} for the inhomogeneous quasilinear case), and stability in minimum time was provided in \cite{Auriol2015MinimumTC}. Following these major developments, \cite{Anfinsen2016} proposed an adaptive observer to estimate the boundary parameters of an $n+1$ system motivated by the need to identify the bottom hole influxes of hydrocarbon caused by high pressure formations in the well during oil drilling operations. The extension of  \cite{Anfinsen2016} to $n+m$ systems in \cite{anfinsen2017estimation} was followed by major progress on adaptive control design \cite{anfinsen2019adaptive}. The recent results in\cite{coron2017finite,coron2021boundary} allow for finite-time stabilization of linear and coupled hyperbolic systems with space and time-dependent parameters. In light of the current context, it's worth noting the emergence of nonlinear controllers for nonlinear infinite-dimension systems of conservation laws \cite{karafyllis2022spill, karafyllis2023output}. These controllers demonstrate the potential for achieving global exponential stabilization of spill-free transfer systems governed by nonlinear and coupled hyperbolic systems.

In general, the conception of PDE controllers can lead to complex gain functions that require non-obvious computational effort. Our contribution signifies an advancement in leveraging the computational capabilities offered by machine learning techniques to enhance the feasibility of hyperbolic PDE control. 
\subsection{Contributions}
We expedite the computation of gain kernel PDEs that emerge from backstepping design for coupled linear hyperbolic systems. Developing further the DeepONet design originally introduced in \cite{bhan2023neural} and then \cite{krstic2023neural,qi2023neural} for simpler PDE systems, we introduce Neural Operator (NO) approximations for kernels applicable to $2\times 2$ hyperbolic PDEs to encapsulate the mapping from the functional coefficients of the plant into a previously trained DeepONet. We design a neural network architecture, or more precisely, a computational resource capable of calculating the gains through function evaluations, eliminating the necessity to solve the coupled gain kernel PDEs defined on a triangular domain. Recently, DeepONet achieved gain kernel computation for full-state feedback control in the ARZ traffic system (simpler case) in \cite{zhang2023neural}. Furthermore, results on DeepONet-based adaptive control \cite{lamarque2024adaptive}, gain scheduling \cite{lamarque2024gain}, and moving-horizon estimators (MHE) \cite{bhan2024moving} were recently developed.

Differing from \cite{wang2023Deep}, where DeepONet approximation of gain kernel PDEs was achieved using a composition of operators defined by a single hyperbolic PDE in Goursat form and one parabolic PDE on a rectangular domain, the scenario involving coupled hyperbolic PDEs in cascade, along with their state observer, gives rise to an approximate closed-loop system governed by four interconnected Goursat-form PDEs. Two of these Goursat-form PDEs originate from the controller, and the remaining two characterize the gain function of the observer. Both the DeepONet approximates of the controller and the observer gain functions are the output of the $2\times2$ coupled nonlinear operators of Goursat PDEs fed by a total of five plant functional coefficients. The configuration of the studied DeepOnet-approximate nonlinear operators expands the scope of NO designs originally introduced by the Machine Learning community \cite{chen1995universal,lu2021learning,lu2019deeponet,li2020fourier} because such a distinctive coupling between Goursat PDEs is primarily established through backstepping control design. Our present contribution is twofold:

\begin{itemize}
\item \textbf{DeepONet for the gain kernels of the output-feedback law.}
We derive a \emph{Global Exponential Stability (GES)} result for a $2\times 2$ linear hyperboblic PDE system equipped with an output feedback control law fed by the NO-approximate controller's and observer's gain functions. Considering that the controller-observer system is a composition of two linear systems, the global exponential stability of the closed-loop system that is achieved by the exact gain functions is preserved under the approximated one with an estimate of the decay rate. The effect of the accuracy of the approximation on the convergence rate is elucidated: less accurate data slows the decay rate.
\item \textbf {Encoding of the output feedback law.} Utilizing insights gained from the DeepONet controller and observer gain kernels, as well as the observed system state values, we develop a neural operator (NO) approximation for the output-feedback control law, including the state of the observer. This case deals with the complete learning of a control law for a $2\times 2$ linear and coupled hyperbolic system using anti-collocated boundary actuation and sensing. We establish a \emph{Semi-global Practical Exponential Stability (SG-PES) } estimate for the resulting closed-loop system. The SG-PES result cannot be attributed to the approximation of Goursat-form PDEs that are multiplicative gain kernels embedded within the control law and the state estimator, but to the approximation of the observer states $u$ and $v$ that are also used into a \emph{fully-learned} control law. As stated concisely, the approximation error is not solely multiplicative but also additive, resulting in the SG-PES outcome. In a nutshell, the stability result is semi-global considering that the dataset includes samples of observer states $u$ and $v$ with bounded magnitudes. 
\end{itemize}

In both cases, the method accelerates the computation of control gains, significantly improving computation speed. Our theoretically proven stability result is demonstrated through simulations, and the code is available on github.

\textbf{\emph{Organization of paper:}} 
Section \ref{sec:problem_statement} succinctly presents the design of an exponentially stabilizing output-feedback boundary control law for $2\times 2$ hyperbolic PDE systems. Sections \ref{sec3} and \ref{Stab-Deep} present the approximation of the kernel operators and the global exponential stabilization (GES) under the approximated controller's gain functions and observer's gain functions via DeepONet. 
Section \ref{Sec5} presents a semi-global practical exponential stability (SG-PES) result when the totality of the output feedback law is learned via DeepONet. Section \ref{Simulation} and Section \ref{conclusion} present our simulation results and concluding remarks, respectively.

\textbf{\emph{Notation:}} We define the $L^2$-norm for $\chi(x)\in L^2[0,1]$ as $\rVert\chi \rVert_{L^2}^2=\int_0^1|\chi(x)|^2dx$. For the convenience, we set $\rVert\chi \rVert^2=\rVert\chi \rVert_{L^2}^2$. The supremum norm is denoted $\rVert \cdot \rVert_\infty$.

\maketitle


\section{\protect Preliminaries and  problem statement}\label{sec:problem_statement}
\textbf{Preliminaries.}
We consider linear hyperbolic  systems 
\begin{align}
\partial_t u(x, t) =&-\lambda(x) \partial_x u(x, t)+\sigma(x) u(x, t) \notag \\
& +\omega(x) v(x, t), \label{eq:sys_u} \\
\partial_t v(x, t)=&\mu(x) \partial_x v(x, t)+ \theta(x) u(x, t),\label{eq:sys_v}
\end{align}
with boundary conditions 
\begin{align}
u(0, t) & = q v(0, t),\label{eq:sys_BC_1} \\
v(1, t) & = U(t),\label{eq:sys_BC_2}
\end{align}
where, 
\begin{align}
\lambda, \mu & \in C^1([0, 1]),\\
\sigma, \omega, \theta & \in C^0([0, 1]),\\
q& \in \mathbbm R,
\end{align}
and initial conditions
\begin{equation}
v^0(x), \; u^0(x) \in  L^2([0, 1]).\label{initial_data}
\end{equation}
The transport speeds are assumed to satisfy
\begin{equation}
-\mu(x) < 0 < \lambda(x), \quad \forall x \in [0, 1],
\end{equation}
and   $\lambda,\ \mu,\ \sigma,\ \omega, \ \theta$ are all bounded with  $\underline \lambda\leq \lambda\leq\bar\lambda$, $\underline \mu\leq \mu\leq\bar\mu$, $\underline \sigma\leq \sigma \leq\bar\sigma$, $\underline \omega\leq \omega\leq\bar\omega$, and $\underline \theta\leq \theta\leq\bar\theta$.

\subsection{Full-state boundary  feedback control law}
Exploiting  the following backstepping transformation \cite{Meglio2013},
\begin{align}
\nonumber \beta(x, t)  =&v(x, t)-\int_0^x k_1(x, \xi) u(\xi, t) d\xi \\
&-\int_0^x k_2(x, \xi) v(\xi, t) d\xi,\label{eq:backstepping_2} 
\end{align}
system \eqref{eq:sys_u}--\eqref{eq:sys_BC_2} can be transformed into the target system 
\begin{align}
\nonumber \partial_t u(x,& t)  = - \lambda(x) \partial_x u(x, t)+\sigma(x) u(x, t) +\omega(x) \beta(x, t)\\
&+\int_0^x c(x,\xi) u(\xi,t)d\xi +\int_0^x \kappa(x,\xi) \beta(\xi,t)d\xi,\label{eq:sys_alpha} \\
\partial_t \beta(x,& t)=\mu(x) \partial_x \beta(x, t),\label{eq:sys_beta}
\end{align}
with boundary conditions defined as
\begin{align}
u(0, t) & = q \beta(0, t) , \label{eq:sys_BC_alpha} \\
\beta(1, t) & =  0,\label{eq:sys_BC_beta}
\end{align}
where $c(x,\xi)$ and $\kappa(x,\xi)$ are functions to be determined. The realization of this mapping requires the kernels in the backstepping transformation \eqref{eq:backstepping_2} to satisfy the following PDEs\footnote{Here, we use the prime notation to indicate derivatives.}
\begin{align}
\mu(x) \partial_x k_1-\lambda(\xi) \partial_\xi k_1
&=\lambda^\prime(\xi) k_1+\sigma(\xi) k_1+\theta(\xi) k_2,\label{eq:k_1} \\
\mu(x) \partial_x k_2+\mu(\xi) \partial_\xi k_2&=-\mu^\prime(\xi) k_2 +\omega(\xi) k_1,\label{eq:k_21}
\end{align}
with boundary conditions
\begin{align}
k_1(x, x)& =-\dfrac{\theta(x)}{\lambda(x) + \mu(x)},\label{eq:k_BC_1} \\
\mu(0)k_2(x, 0)& = q\lambda(0)k_1(x,0).\label{eq:k_BC_2}
\end{align}
The system \eqref{eq:k_1}--\eqref{eq:k_BC_2} defined over the triangular domain $ \mathcal T = \{ (x, \xi) \; | \; 0 \leq \xi \leq x \leq 1 \}$, is a coupled $2\times 2$ Goursat-form PDEs govern two gain kernels and the coefficient $\kappa$ and $c$ are chosen to satisfy
\begin{align}
\kappa(x,\xi)&=\omega(x)k_2(x,\xi)+\int_\xi^x\kappa(x,s)k_2(s,\xi)ds,\label{kap}\\
c(x,\xi)&=\omega(x)k_1(x,\xi)+\int_\xi^x\kappa(x,s)k_1(s,\xi)ds.\label{c}
\end{align}
From \eqref{eq:sys_BC_2}, \eqref{eq:backstepping_2}, and  \eqref{eq:sys_BC_beta}, the boudary controller is 
\begin{align}
U(t)=\int_0^1 k_1(1, \xi) u(\xi, t) d\xi +\int_0^1 k_2(1, \xi) v(\xi, t) d\xi.\label{equ-U}
\end{align}
The invertibility of the transformation \eqref{eq:backstepping_2} together with the existence of unique solution to \eqref{eq:k_1}--\eqref{eq:k_BC_2} is established in \cite{Meglio2013}. The invertibility of the transformation induces equivalent stability properties of the target and original systems.
                
The inverse transformation of \eqref{eq:backstepping_2} is given by
\begin{align}
\nonumber v(x, t)  =&\beta(x, t)+\int_0^x l_1(x, \xi) u(\xi, t) d\xi \\
&+\int_0^x l_2(x, \xi) \beta(\xi, t) d\xi, \label{eq:backstepping_2-inverse}
\end{align}
where
\begin{align}
l_1(x, \xi) & =k_1(x,\xi)+\int_\xi^x k_2(x, s) l_1(s, \xi) ds,\label{eq:kernel-inverse-1}\\
l_2(x, \xi) & =k_2(x,\xi)+\int_\xi^x k_2(x, s) l_2(s,\xi) ds.\label{eq:kernel-inverse-2}
\end{align} 
\subsection{Observer design for an output feedback control law}
Our goal is to develop an exponentially convergent observer capable of estimating the spatially distributed states of system \eqref{eq:sys_u}--\eqref{eq:sys_BC_2} using the available boundary point measurement $v(0,t)$, which is anti-collocated with the boundary point of actuation. We design an  observer consisting of the copy of the plant plus some output injection terms. The derivation of the observer's gains that ensure convergence of the estimated states to the plant states is completed using backstepping design. The following observer is stated: 
\begin{align}
\partial_t \hat u(x, t) =&-\lambda(x) \partial_x \hat u(x, t)+\sigma(x)\hat  u(x, t)+\omega(x)\hat  v(x, t) \notag \\
& +p_1(x)(v(0,t)-\hat v(0,t)), \label{eq:sys_u-obs} \\
\nonumber \partial_t \hat v(x, t)=&\mu(x) \partial_x \hat v(x, t)+ \theta(x)\hat  u(x, t)\\
&+p_2(x)(v(0,t)-\hat v(0,t)),\label{eq:sys_v-obs}
\end{align}
with boundary conditions 
\begin{align}
\hat u(0, t) & = q v(0, t),\label{eq:sys_BC_1-obs} \\
\hat v(1, t) & = U(t).\label{eq:sys_BC_2-obs}
\end{align}
The functions $p_1(x)$ and $p_2(x)$ are the observer output injection gains to be determined via backstepping design. 
Denoting the observer error 
\begin{align}
\tilde u(x,t)&=u(x,t)-\hat u(x,t),\\
\tilde v(x,t)&=v(x,t)-\hat v(x,t),
\end{align}
it follows the error dynamics
\begin{align}
\partial_t \tilde u(x, t) =&-\lambda(x) \partial_x \tilde u(x, t)+\sigma(x)\tilde u(x, t) \notag \\
& +\omega(x)\tilde  v(x, t)-p_{1}(x)\tilde v(0,t), \label{eq:sys_u-err} \\
\partial_t \tilde v(x, t)=&\mu(x) \partial_x \tilde v(x, t)+ \theta(x)\tilde  u(x, t)-p_{2}(x)\tilde v(0,t),\label{eq:sys_v-err}
\end{align}
with boundary conditions 
\begin{align}
\tilde u(0, t) & =0,\label{eq:sys_BC_1-err} \\
\tilde v(1, t) & =0.\label{eq:sys_BC_2-err}
\end{align}
To design the observer output injection gains, backstepping transformations are again introduced as 
\begin{align}
\tilde u(x, t) & =\tilde \alpha(x,t)+\int_0^x m_1(x, \xi) \tilde \beta(\xi,t) d\xi,\label{eq:kernel-inverse-m1}\\
\tilde v(x, t)& =\tilde \beta (x,t)+\int_0^x m_2(x, \xi) \tilde \beta(\xi,t) d\xi,\label{eq:kernel-inverse-m2}
\end{align}
to map system \eqref{eq:sys_u-err}--\eqref{eq:sys_BC_2-err} into the target system
\begin{align}
\partial_t \tilde \alpha(x, t) =&-\lambda(x) \partial_x \tilde \alpha(x, t)+\sigma(x)\tilde \alpha(x, t) \notag \\
& +\int_0^xg(x,\xi)\tilde\alpha(\xi,t)\mathrm d\xi,\label{eq:target_u-obs} \\
\nonumber\partial_t \tilde \beta(x, t)=&\mu(x) \partial_x \tilde \beta(x, t)+ \theta(x)\tilde  \alpha(x, t)\\
&+\int_0^xh(x,\xi)\tilde\alpha(\xi,t)\mathrm d\xi,\label{eq:target_v-obs}
\end{align}
with boundary conditions 
\begin{align}
\tilde \alpha (0, t)  =0, \quad
\tilde \beta(1, t) =0,\label{eq:target_BC_2-obs}
\end{align}
where $g(x,\xi)$ and $h(x,\xi)$ are functions to be determined. The stability of the target system  governed by \eqref{eq:target_u-obs}--\eqref{eq:target_BC_2-obs} and  in relation to that of the error dynamics \eqref{eq:sys_u-err}--\eqref{eq:sys_BC_2-err} was demonstrated  in Lemma 3.3 \cite{Meglio2013}.

To transform \eqref{eq:sys_u-err}--\eqref{eq:sys_BC_2-err} into \eqref{eq:target_u-obs}--\eqref{eq:target_BC_2-obs}, the kernels of \eqref{eq:kernel-inverse-m1} and \eqref{eq:kernel-inverse-m2} must adhere to the following PDEs
\begin{align}
\lambda(x) \partial_x m_1
-\mu(\xi) \partial_\xi m_1&=\mu^\prime(\xi)m_1+\delta(x) m_1+ \omega(x)m_2,\label{eq:m_1} \\
\mu(x) \partial_x m_2+\mu(\xi) \partial_\xi m_2&=-\mu^\prime(\xi) m_2 -\theta(x) m_1,\label{eq:m_21}
\end{align}
with boundary conditions
\begin{align}
m_1(x, x)& =\dfrac{\omega(x)}{\lambda(x) + \mu(x)},\label{eq:m_BC_1} \\
m_2(1, \xi)& =0.\label{eq:m_BC_2}
\end{align}
The kernels  PDEs, as specified in \eqref{eq:m_1}--\eqref{eq:m_BC_2} is defined over the triangular domain $ \mathcal T = \{ (x, \xi) \; | \; 0 \leq \xi \leq x \leq 1 \}$. It is a coupled \emph{$2\times 2$ Goursat-form PDEs} governing two  kernels equations and the coefficient $g$ and $h$ are specified such that
\begin{align}
g(x,\xi)&=-\theta(\xi)m_1(x,\xi)-\int_\xi^xm_{1}(x,s)h(s,\xi)ds,\\
h(x,\xi)&=-\theta(\xi)m_{2}(x,\xi)-\int_\xi^xm_{2}(x,s)h(s,\xi)ds.
\end{align}
The invertibility of  the transformation \eqref{eq:backstepping_2} together with the existence of unique solution to \eqref{eq:k_1}--\eqref{eq:k_BC_2} is established in \cite{Meglio2013}. It consequently  implies equivalent  stability properties of the target system  \eqref{eq:target_u-obs}--\eqref{eq:target_BC_2-obs} and the original system \eqref{eq:sys_u-err}--\eqref{eq:sys_BC_2-err}, and 
\begin{align}
p_1(x)=&m_1(x,0)\mu(0),\\
p_2(x)=&m_2(x,0)\mu(0).
\end{align}          
The inverse transformation of \eqref{eq:kernel-inverse-m2} is given by
\begin{align}
\tilde \beta(x,t)=\tilde v(x,t)+\int_0^xr_2(x,\xi)\tilde v(\xi,t)\mathrm d\xi,\label{equ-inverse-m2}
\end{align}
where $r_2(x,\xi)$ satisfies
\begin{align}
r_2(x, \xi) & =-m_2(x,\xi)-\int_\xi^x m_2(x, s) r_2(s,\xi) ds.\label{eq:kernel-inverse-2-r2}
\end{align}
The substitution of  \eqref{equ-inverse-m2} into \eqref{eq:kernel-inverse-m1} results in \begin{align}
\nonumber\tilde \alpha(x, t) =&\tilde u(x,t)-\int_0^x m_1(x, \xi) \bigg(\tilde v(\xi,t)\\
\nonumber&+\int_0^\xi r_2(\xi,s)\tilde v(s,t)\mathrm ds\bigg)\mathrm d\xi\\
\nonumber=&\tilde u(x,t)-\int_0^x m_1(x, \xi) \tilde v(\xi,t)\mathrm d\xi\\
\nonumber&-\int_0^x \int_\xi^x m_1(x,s)r_2(s,\xi)\mathrm ds\tilde v(\xi,t)\mathrm d\xi\\
=&\tilde u(x,t)+\int_0^x r_1(x, \xi) \tilde v(\xi,t)\mathrm d\xi,\label{equ-inverse-m1}
\end{align}
where
\begin{align}
r_1(x, \xi) & =-m_1(x, \xi)-\int_\xi^x m_1(x, s) r_2(s, \xi) ds.\label{eq:kernel-inverse-2-r1}
\end{align}

The exponential stability of the target system governed by \eqref{eq:target_u-obs}--\eqref{eq:target_BC_2-obs} and to that of the error dynamics \eqref{eq:sys_u-err}--\eqref{eq:sys_BC_2-err} is stated in Lemma 3.3 \cite{Meglio2013}. The invertibility of the transformation  implies the  global exponential convergence  of  the error system in $L^2$ sense  \eqref{eq:sys_u-err}--\eqref{eq:sys_BC_2-err} and the $L^2$-Global Exponential Stability of the plant \eqref{eq:sys_u}--\eqref{eq:sys_BC_2} combined with the observer \eqref{eq:sys_u-obs}--\eqref{eq:sys_BC_1-obs} and subject to the control law 
\begin{align}
U(t)=\int_0^1 k_1(1, \xi) \hat u(\xi, t) d\xi +\int_0^1 k_2(1, \xi) \hat v(\xi, t) d\xi.\label{equ-hatU}
\end{align}
We refer the reader to \cite{Meglio2013} for more details about the design of the output feedback law \eqref{equ-hatU}, turning our attention to the DeepOnet designs for the output feedback law.

\textbf{Problem statement.} As shown in Figures \ref{Operator} and \ref{Learning-process-kernel}, our goal is to design NOs to learn the controller and observer gain functions governed by \eqref{eq:k_1}--\eqref{eq:k_BC_2} and \eqref{eq:m_1}--\eqref{eq:m_BC_2}, respectively.  The plant function coefficients are the inputs of the nonlinear operators defined by these hyperbolic/Goursat PDEs. We first aim to prove a DeepONet approximation to the kernel PDEs by showing the continuity of the mapping from plant PDE coefficients to kernel PDE solutions.
The second part of our design consists of the DeepONet encoding of the output-feedback law. Proof-based machine learning designs are presented in this paper.


\section{Accuracy of Approximation of Backstepping Kernel Operator with DeepONet }\label{sec3}

\subsection{Boundedness of the gain kernel functions}

\begin{lemm}\label{lem1}
{\bf\em }
For every $\lambda, \mu  \in C^1([0, 1]),\ 
\sigma, \omega, \theta  \in C^0([0, 1])$, and $
q\in \mathbbm R$, the gain kernels $k_i(x,\xi), \textcolor{black}{m_i(x,\xi),}\ i=1,2$ satisfying the PDE systems \eqref{eq:k_1}--\eqref{eq:k_BC_2} \textcolor{black}{and \eqref{eq:m_1}--\eqref{eq:m_BC_2}, respectively}, has a unique $C^1(\mathcal T)$ solution with the following  property
\textcolor{black}{\begin{align}
|k_i(x,\xi)|\leq& N_i {\rm e}^{M_i},\quad i=1,2,\quad\forall(x,\xi)\in\mathcal T,
\label{equ-k-bouded}\\
|m_i(x,\xi)|\leq& N_i {\rm e}^{M_i},\quad i=1,2,\quad\forall(x,\xi)\in\mathcal T,
\label{equ-m-bouded}
\end{align}}
where $N_i>0,\ M_i>0,\ i=1,2$ are constants.
\end{lemm}

\begin{proof}\rm
  The proof of Lemma \ref{lem1} can be found in  \cite{Meglio2013}.
\mbox{}\hfill
\end{proof}
\subsection{Approximation of the neural operators}

We pursue the characterization of  the neural operators  to learn the gain kernels function-to-function mapping using finite input-output pairs of collected data. Recall that based on the gain kernels PDEs  \eqref{eq:k_1}--\eqref{eq:k_BC_2}, 
a neural approximation of the  operator  $(\lambda,\mu,\omega,\sigma,\theta,q)\mapsto (k_1,\ k_2,\ \textcolor{black}{m_1,\ m_2})$ is describe by   Figure \ref{Operator}. 
\begin{figure}[t]
\centering
\includegraphics[width=0.405\textwidth]{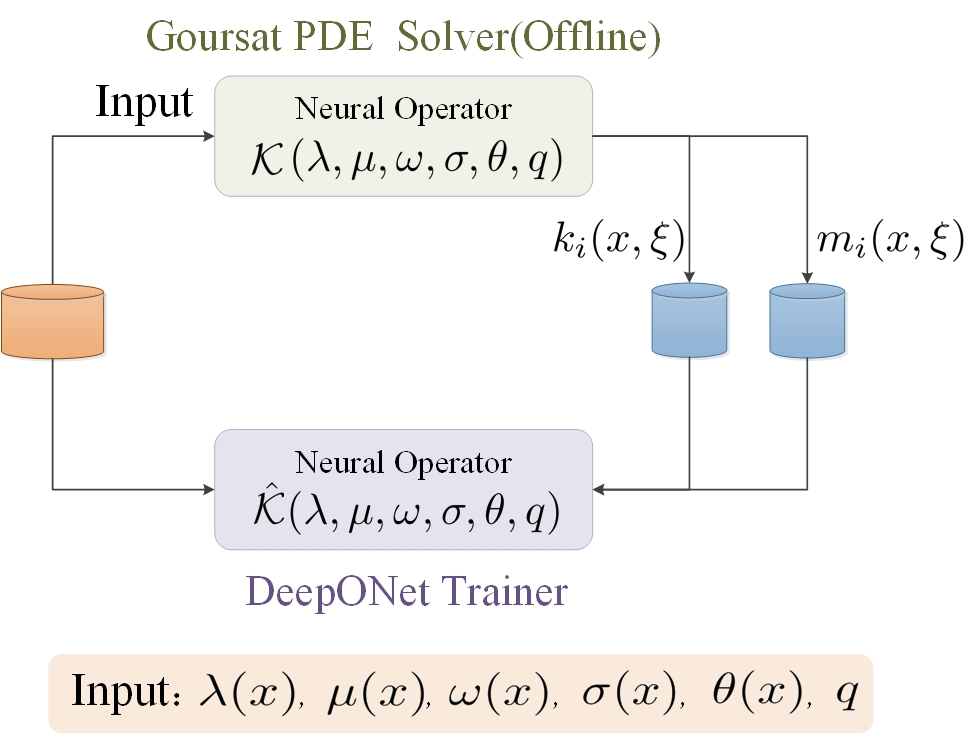}
\vspace{-0.15cm}
\caption{Learning of the kernel functions via DeepONet and through the operator  described by the mapping $(\lambda,\mu,\omega,\sigma,\theta,q)\to (k_1,k_2,m_1,m_2)$. 
Computing multiple solutions of kernel PDEs \eqref{eq:k_1}--\eqref{eq:k_BC_2} in the Goursat form for different functions $\lambda(x),\ \mu(x),\ \omega(x),\ \sigma(x),\ \theta(x)$ and parameters $q$,  completes the training procedure of the Neural Operator $\hat{\mathcal{K}}$.} \label{Operator}
\end{figure}

\textcolor{black}{Define the operator ${\mathcal{K}}: (C^1[0,1])^2\times (C^0[0,1])^3\times  \mathbbm R\mapsto (C^1(\mathcal T))^{4}$, where
\begin{align}
\nonumber&{\mathcal{K}}(\lambda, \mu,  \sigma, \omega, \theta,
q)(x,\xi)\\
:=&(k_{1}(x,\xi),k_{2}(x,\xi),m_{1}(x,\xi),m_{2}(x,\xi)),\label{neur-op2}
\end{align}
allowing to introduce the operator  ${\mathcal{M}}:(C^1[0,1])^2\times (C^0[0,1])^3\times  \mathbbm R\mapsto (C^1(\mathcal T))^{2}\times (C^0(\mathcal T))^{4}\times(C^0[0,1])^{2}\times (C^1(\mathcal T))^{4}$  defined by 
\begin{align}
\nonumber&{\mathcal{M}}(\lambda, \mu,  \sigma, \omega, \theta,
q)(x,\xi)\\
:=&(k_{1},k_{2},c,\kappa,K_1, K_2,K_3,K_4,K_5,K_6,K_7,K_8)\,,\label{equ-M-1}
\end{align}
where
\begin{align}
K_1(x)=&(\lambda(x) + \mu(x))k_1(x, x)+\theta(x),\label{equ-K1}\\
K_2(x)=&-\lambda(0)qk_1(x, 0)+\mu(0)k_2(x, 0),\\
\nonumber K_3(x,\xi)=&-\mu(x) \partial_x k_1+\lambda(\xi) \partial_\xi k_1+\lambda^\prime(\xi) k_1+\sigma(\xi) k_1\\
&+\theta(\xi) k_2,\\
K_4(x,\xi)=&-\mu(x) \partial_x k_2-\mu(\xi) \partial_\xi k_2-\mu^\prime(\xi) k_2 +\omega(\xi) k_1,\label{equ-K4}\\
\nonumber K_5(x,\xi)=&-\lambda(x) \partial_x m_1+\mu(\xi) \partial_\xi m_1-\mu^\prime(\xi) m_1+\sigma(\xi) m_1\\
&+\omega(x) m_2,\\
K_6(x,\xi)=&\mu(x) \partial_x m_2+\mu(\xi) \partial_\xi m_2+\mu^\prime(\xi) m_2 +\theta(\xi) m_1,\\
K_7(x)=&m_1(x,x)(\lambda(x)+\mu(x))-\omega(x),\\
K_8(\xi)=&m_2(1,\xi),\label{equ-K6}
\end{align}
is introduced. The operators ${\mathcal{K}}$ and ${\mathcal{M}}$ are useful to state the following theorem. }

\medskip\begin{theorem}
\label{thm-karniadakis-bkst}
{\bf\em [DeepONet approximation of kernels]}

Consider the neural operator defined in \eqref{equ-M-1}, along with \eqref{equ-K1}--\eqref{equ-K4} and let $B_\lambda,\ B_\mu,\ B_\sigma,\ B_\omega,\ B_\theta,\ B_{\lambda'},\ B_{\mu'}>0$ be arbitrarily large and $\epsilon>0$, there exists a neural operator $\hat{\mathcal{M}}:(C^1[0,1])^2\times (C^0[0,1])^3\times  \mathbbm R\mapsto (C^1(\mathcal T))^{2}\times (C^0(\mathcal T))^{2}\times(C^0[0,1])^{2}\times (C^1(\mathcal T))^{4}$ such that, 
\begin{align}
&|{\mathcal{M}}(\lambda, \mu,  \sigma, \omega, \theta,
q)(x,\xi)-\hat{\mathcal{M}}(\lambda, \mu,  \sigma, \omega, \theta,
q)(x,\xi)|<\epsilon,
\end{align}
holds for all Lipschitz  $\lambda,\ \mu,\
\sigma,\ \omega,\ \theta$ and derivations of $\lambda, \mu$  with properties that $\rVert \lambda \rVert_\infty  \leq B_\lambda,\ \rVert \mu \rVert_\infty  \leq  B_\mu,\ \rVert \sigma \rVert_\infty  \leq B_\sigma,\ \rVert \omega \rVert_\infty  \leq B_\omega,\ \rVert \theta \rVert_\infty  \leq B_\theta,\ \rVert \lambda' \rVert _\infty \leq B_{\lambda'},\ \rVert \mu' \rVert_\infty  \leq  B_{\mu'}$, namely, 
there exists a neural operator $\hat {\mathcal{K}}$ such that
\begin{align}\label{equ-epsilon}
\nonumber&|\tilde k_1(x,\xi)|+|\tilde k_2(x,\xi)|+|\tilde c(x,\xi)|+|\tilde \kappa(x,\xi)|\\
\nonumber&+|(\lambda(x) + \mu(x))\tilde k_1(x, x)|+|\lambda(0)q\tilde k_1(x, 0)-\mu(0)\tilde k_2(x, 0)| \\
\nonumber&+|-\mu(x) \partial_x\tilde  k_1+\lambda(\xi) \partial_\xi \tilde k_1+\lambda^\prime(\xi)\tilde  k_1+\sigma(\xi)\tilde  k_1+\theta(\xi) \tilde k_2|\\
\nonumber&+|-\mu(x) \partial_x\tilde  k_2-\mu(\xi) \partial_\xi \tilde k_2-\mu^\prime(\xi)\tilde  k_2 +\omega(\xi)\tilde  k_1|\\
\nonumber&~\textcolor{black}{+|\lambda(x) \partial_x \tilde m_1-\mu(\xi) \partial_\xi \tilde m_1+\mu^\prime(\xi)\tilde m_1-\sigma(\xi) \tilde m_1}\\
\nonumber &~\textcolor{black}{-\omega(x) \tilde m_2|+|\mu(x) \partial_x \tilde m_2+\mu(\xi) \partial_\xi \tilde m_2+\mu^\prime(\xi) \tilde m_2}\\
&~\textcolor{black}{+\theta(\xi) \tilde m_1|}<\epsilon,
\end{align}
where $\tilde c(x,\xi)=c(x,\xi)-\hat c(x,\xi)$, $\tilde \kappa(x,\xi)=\kappa(x,\xi)-\hat \kappa(x,\xi)$, and 
\begin{align}
\tilde k_i(x,\xi)=&k_i(x,\xi)-\hat k_i(x,\xi),\quad i=1,2,\\
\tilde m_i(x,\xi)=&m_i(x,\xi)-\hat m_i(x,\xi),\quad i=1,2,
\end{align}
and 
\begin{align}
\nonumber&(\hat k_1(x,\xi),\hat k_2(x,\xi), \textcolor{black}{\hat m_1(x,\xi),\hat m_2(x,\xi)})\\
=&\hat {\mathcal{K}}(\lambda, \mu,  \sigma, \omega, \theta, q)(x,\xi). 
\end{align}
\end{theorem}

\begin{proof}\rm
    The continuity of the operator $\mathcal{M}$ follows from Lemma \ref{lem1}. The result is obtained by invoking \cite[Thm. 2.1]{lu2021advectionDeepONet}. 
\mbox{}\hfill
\end{proof}

\begin{figure}[t]
\centering
\includegraphics[width=0.405\textwidth]{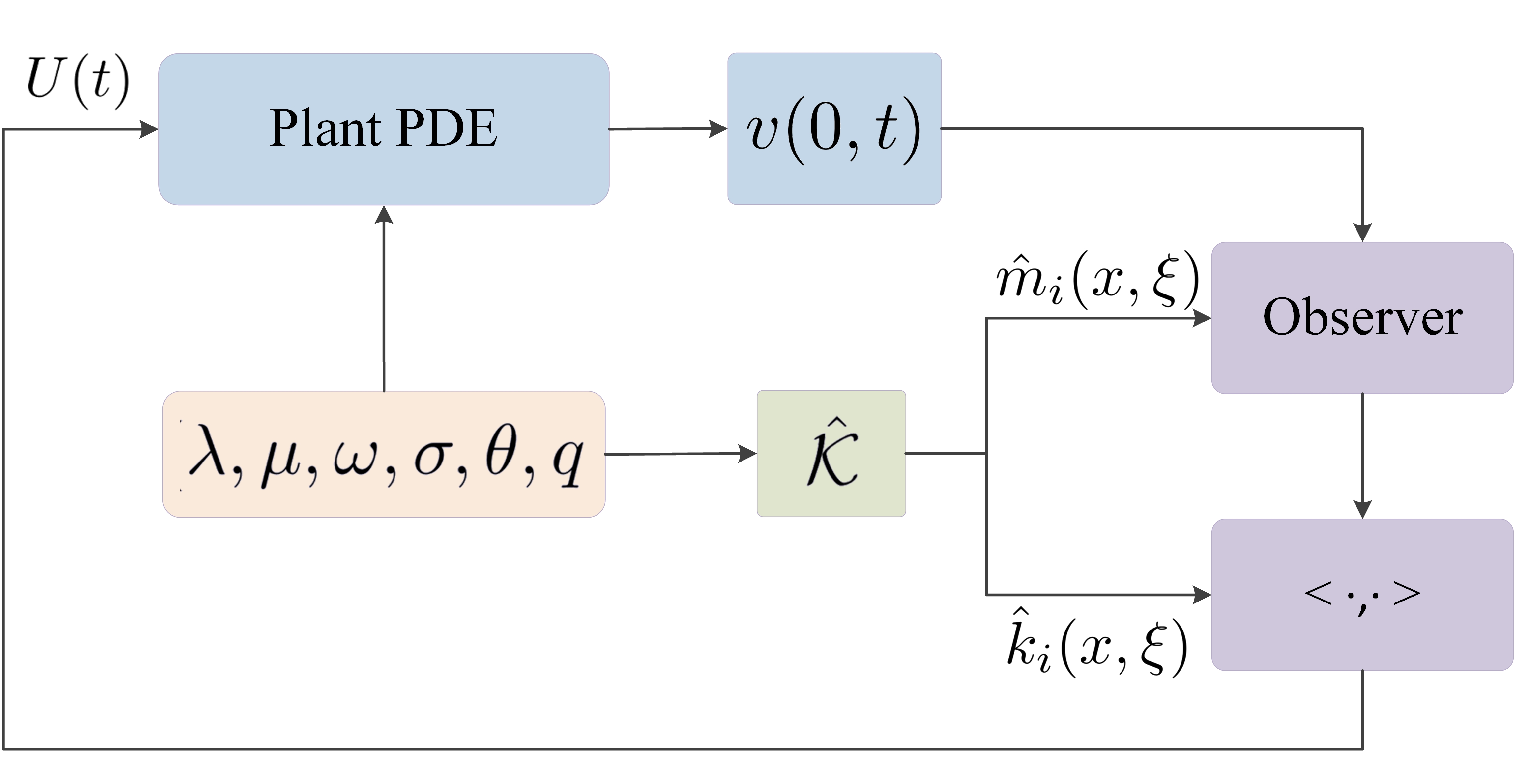}
\vspace{-0.15cm}
\caption{The PDE backstepping observer \eqref{eq:sys_u-obs}--\eqref{eq:sys_BC_2-obs} uses boundary measurement of the flux $v(0,t)$. The gains $\hat k_i$ and $\hat m_i$, $i=1,2$ are produced with the DeepONet $\hat {\mathcal{K}}$.} \label{Learning-process-kernel}
\end{figure}

\section{Stabilization under DeepONet Gain \textcolor{black}{Output} Feedback}\label{Stab-Deep}
The schematic of the control loop is depicted in Figure \ref{Learning-process-kernel}. We will prove that the gain, a priori learned from the DeepOnet layer (offline), enforces closed-loop system stability with a quantifiable exponential decay rate.
The system  in closed-loop form, consists of the observer  \eqref{eq:sys_u-obs}--\eqref{eq:sys_BC_2-obs} in combination with the output-feedback boundary control law 
\begin{align}
U(t)  =&\int_0^1 \hat k_1(1, \xi) \hat u(\xi, t) d\xi +\int_0^1 \hat k_2(1, \xi)\hat v(\xi, t) d\xi.\label{eq:U-hat} 
\end{align}
The backstepping transformations \eqref{eq:backstepping_2}, \eqref{eq:kernel-inverse-m1} and  \eqref{eq:kernel-inverse-m2} fed by the approximate gain kernels $\hat k_i(x,\xi)$ and $\hat m_i(x,\xi)$, $i=1,2$, are defined as
\begin{align}
\hat u(x,t)=&\hat u(x,t),\label{equ-trans-wu}\\
\nonumber\hat z(x, t)=&\hat v (x,t)-\int_0^x \hat k_1(x, \xi) \hat u(\xi,t) d\xi\\
&-\int_0^x \hat k_2(x, \xi) \hat v(\xi,t) d\xi,\label{eq:kernel-inverse-beta-hat}\\
\tilde u(x, t) =&\tilde w(x,t)+\int_0^x \hat m_1(x, \xi) \tilde z(\xi,t) d\xi,\label{eq:kernel-inverse-m1-hat}\\
\tilde v(x, t) =&\tilde z (x,t)+\int_0^x \hat m_2(x, \xi) \tilde z(\xi,t) d\xi,\label{eq:kernel-inverse-m2-hat}
\end{align}
where 
\begin{align}    
&(\hat k_{1}(x,\xi),\hat k_{2}(x,\xi),\hat m_{1}(x,\xi),\hat m_{2}(x,\xi))\nonumber\\ 
=&\hat {\mathcal{K}}(\lambda, \mu,  \sigma, \omega, \theta,
q)(x,\xi).
\end{align} 
Exploiting transformations \eqref{eq:kernel-inverse-beta-hat}--\eqref{eq:kernel-inverse-m2-hat},  the conversion of the observer system \eqref{eq:sys_u-obs}--\eqref{eq:sys_BC_2-obs} and  error system \eqref{eq:sys_u-err}--\eqref{eq:sys_BC_2-err}  into a cascaded target system below is achieved
\begin{align}
\nonumber&\partial_t\hat u(x,t)=-\lambda(x) \partial_x \hat u(x,t)+\sigma(x)\hat u(x,t)+\omega(x)\hat z(x,t)\\
\nonumber&~~~~~~+\int_0^x \hat c(x,\xi)\hat u(\xi,t)d\xi +\int_0^x \hat \kappa(x,\xi)\hat z(\xi,t)d\xi\\
&~~~~~~~~+\hat m_1(x,0)\mu(0)\tilde z(0,t),\label{equ-w0-hat}\\
 &\partial_t\hat z(x,t)=\mu(x) \partial_x \hat z(x,t)+\delta_{1}(x)\hat u(x,t)+\delta_{2}(x)\hat z(0,t)\nonumber\\
\nonumber&~~~~~~+\int_0^x \delta_{3}(x,\xi) \hat u(\xi,t)d\xi +\int_0^x \delta_{4}(x,\xi)\hat v(\xi,t)d\xi\\
&~~~~~~+F(x)\mu(0)\tilde z(0,t),\\
&\hat u(0,t)=q\hat z(0,t),\\
&\hat z(1,t)=0,\label{equ-z1}\\
\nonumber&\partial_t \tilde w(x,t)=-\lambda(x) \partial_x \tilde w(x,t)+\sigma(x)\tilde w(x,t)\\
\nonumber&~~~~~~+\int_0^x \hat g(x,\xi) \tilde w(\xi,t)\mathrm d\xi+\int_0^x \delta_5(x,\xi) \tilde z(\xi,t)\mathrm d\xi \\
&~~~~~~ +\int_0^x\int_\xi^x \hat r_1(x,s)\delta_6(s,\xi)\mathrm ds\tilde z(\xi,t)\mathrm d\xi,\\
 &\partial_t\tilde z(x,t)=\mu(x) \partial_x \tilde z(x,t)+\theta(x)\tilde w(x,t)\nonumber\\
\nonumber&~~~~~~+\int_0^x \hat h(x,\xi) \tilde w(\xi,t)\mathrm d\xi+\int_0^x \delta_6(x,\xi) \tilde z(\xi,t)\mathrm d\xi \\
&~~~~~~ +\int_0^x\int_\xi^x \hat r_2(x,s)\delta_6(s,\xi)\mathrm ds\tilde z(\xi,t)\mathrm d\xi,\\
&\tilde w(0,t)=0,\\
&\tilde z(1,t)=0,\label{equ-z1-hat}
\end{align}
where
\begin{align}
\nonumber F(x)&=\hat m_2(x,0)-\int_0^x\hat k_1(x,\xi)\hat m_1(\xi,0)\mathrm d\xi\\
&-\int_0^x\hat k_2(x,\xi)\hat m_2(\xi,0)\mathrm d\xi,\label{equ-F}\\
\hat g(x,\xi)&=-\theta(\xi)\hat m_1(x,\xi)-\theta(\xi)\int_\xi^x\hat m_1(x,s)\hat r_2(s,\xi)ds,\label{hat-g}\\
\hat h(x,\xi)&=-\theta(\xi)\hat m_2(x,\xi)-\theta(\xi)\int_\xi^x\hat m_2(x,s)\hat r_2(s,\xi)ds,\label{hat-h}\\
\hat \kappa(x,\xi)&=\omega(x)\hat k_2(x,\xi)+\int_\xi^x\hat \kappa(x,s)\hat k_2(s,\xi)ds,\label{hat-kap}\\
\hat c(x,\xi)&=\omega(x)\hat k_1(x,\xi)+\int_\xi^x\hat c(x,s)\hat k_1(s,\xi)ds,\label{hat-c}
\end{align}
and  $\delta_i,\ i=1,2,...,6$ are defined as 
\begin{align}
\delta_1(x)&=(\lambda(x)+\mu(x))\tilde k_1(x,x),\label{delta-1}\\
\delta_2(x)&=\lambda(0)q\tilde k_1(x,0)-\mu(0)\tilde k_2(x,0),\label{delta-2}\\
\delta_3(x,\xi)&=\lambda(\xi)'\tilde k_{1}(x,\xi)+\sigma(\xi)\tilde k_{1}(x,\xi)+\theta(\xi)\tilde k_{2}(x,\xi)\nonumber
\\
&-\mu(x) \partial_x \tilde k_{1}(x,y)+\lambda(\xi) \partial_\xi \tilde k_{1}(x,\xi),\label{delta-3}\\
\delta_4(x,\xi)&=-\mu(x) \partial_x \tilde k_{2}(x,y)-\mu(\xi)\partial_\xi \tilde k_{2}(x,\xi)\nonumber\\
&-\mu(\xi)'\tilde k_{2}(x,\xi)+\omega(\xi)\tilde k_{1}(x,\xi),\label{delta-4}\\
\delta_5(x,\xi)&=\lambda(x) \partial_x \tilde m_{1}(x,\xi)-\mu(\xi)\partial_\xi m_1(x,\xi)-\sigma(x)  \nonumber\\
&\cdot \tilde m_{1}(x,\xi)-\omega(x)m_2(x,\xi)+\mu(\xi)'\tilde m_{1}(x,\xi),\label{delta-5}\\
\delta_6(s,\xi)&=-\mu(s) \partial_s \tilde m_{2}(s,\xi)-\mu(\xi)\partial_\xi \tilde m_{2}(s,\xi)\nonumber\\
&-\mu(\xi)'\tilde m_{2}(s,\xi)-\theta(s)\tilde m_{1}(s,\xi).\label{delta-6}
\end{align}
Note that from \eqref{equ-epsilon},  the following inequalities hold
\begin{align}
&\rVert\delta_i\rVert_\infty\leq \epsilon,\quad i=1,2,...,6.
\end{align}
We are now poised to state the exponential stability of the target system \eqref{equ-w0-hat}--\eqref{equ-z1-hat} in its backstepping-transformed variables; under the DeepONet-approximated kernels.

\medskip\begin{propo}\label{prop-0}
{\bf\em [Lyapunov analysis for DeepONet-perturbed target system]} 
Consider the cascaded target system \eqref{equ-w0-hat}--\eqref{equ-z1-hat}, there exists $\epsilon^*>0$ such that for all $\epsilon\in (0,\epsilon^*)$, the following holds,
\begin{align}\label{eq-Psi-exp-bound1}
\Psi_1(t)\leq \Psi_1(0)\vartheta_2 {\rm e}^{-\vartheta_1(\epsilon) t}, \quad \forall \geq 0,
\end{align}
where  $\vartheta_1,\ \vartheta_2>0$ and 
\begin{align}
\Psi_1(t)=\rVert \hat u(t)\rVert^2+\rVert \hat z (t)\rVert^2+\rVert\tilde w(t)\rVert^2+\rVert\tilde z (t)\rVert^2.\label{equ-Psi}
\end{align}
\end{propo}


\begin{proof}\rm
As  a generalization of  the proof in \cite{Meglio2013},  we argue that the following Lyapunov candidate for the target system  \eqref{equ-w0-hat}--\eqref{equ-z1-hat}
\begin{align}
\nonumber &V_1(t) =\int_0^1 \frac{\varrho_1e^{-\varrho_2 x}}{\lambda(x)}\hat u (x, t)^2   dx +\int_0^1  \frac{e^{\varrho_2x}}{\mu(x)}\hat z(x, t)^2 dx\\
  &+\int_0^1 \frac{\varrho_3e^{-\varrho_4x}}{\lambda(x)}\tilde w (x, t)^2   dx +\int_0^1  \frac{\varrho_{5}e^{\varrho_4 x}}{\mu(x)}\tilde z(x, t)^2 dx,\label{eq:Lyap_V1}
\end{align}
where $ \varrho_i>0,\ i=1,2,..,5$ are constants to be decided, provides stability at an exponential decay rate to be determined as well. 

Computing the   time derivative of \eqref{eq:Lyap_V1} along \eqref{equ-w0-hat}--\eqref{equ-z1-hat} as
\begin{align}
\nonumber\dot V_1(t)
\nonumber=&2\int_0^1 \frac{\varrho_1e^{-\varrho_2x}}{\lambda(x)} \hat u(x, t)\bigg(-\lambda(x) \partial_x \hat u(x,t)\\
\nonumber&+\sigma(x)\hat u(x,t)+\omega(x)\hat z(x,t)\\
\nonumber&+\int_0^x \hat c(x,\xi)\hat u(\xi,t)d\xi +\int_0^x \hat \kappa(x,\xi)\hat z(\xi,t)d\xi\\
&+\hat m_1(x,0)\mu(0)\tilde z(0,t)\bigg)+2\int_0^1  \frac{e^{\varrho_2x}}{\mu(x)}\hat z(x, t)\nonumber\\
\nonumber&\cdot\bigg(\mu(x) \partial_x \hat z(x,t)+\delta_{1}(x)\hat u(x,t)\\
\nonumber&+\delta_{2}(x)\hat z(0,t)+\int_0^x \delta_{3}(x,\xi) \hat u(\xi,t)d\xi \nonumber\\
\nonumber&+\int_0^x \delta_{4}(x,\xi)\hat v(\xi,t)d\xi+F(x)\mu(0)\tilde z(0,t)\bigg) dx\\
\nonumber&+2\int_0^1 \frac{\varrho_3e^{-\varrho_4x}}{\lambda(x)} \tilde w(x, t)\bigg(-\lambda(x) \partial_x \tilde w(x,t)\\
\nonumber&+\sigma(x)\tilde w(x,t)+\int_0^x \hat g(x,\xi) \tilde w(\xi,t)d\xi \\
\nonumber&+\int_0^x \delta_5(x,\xi) \tilde z(\xi,t)\mathrm d\xi \\
&+\int_0^x\int_\xi^x \hat r_1(x,s)\delta_6(s,\xi)\mathrm ds\tilde z(\xi,t)\mathrm d\xi\bigg)\nonumber\\
\nonumber&+2\int_0^1  \frac{\varrho_5e^{\varrho_4x}}{\mu(x)}\tilde z(x, t)\bigg(\mu(x) \partial_x \tilde z(x,t)+\theta(x)\tilde w(x,t)\nonumber\\
\nonumber&+\int_0^x \hat h(x,\xi) \tilde w(\xi,t)\mathrm d\xi+\int_0^x \delta_6(x,\xi) \tilde z(\xi,t)\mathrm d\xi \\
& +\int_0^x\int_\xi^x \hat r_2(x,s)\delta_6(s,\xi)\mathrm ds\tilde z(\xi,t)\mathrm d\xi\bigg) dx,
\end{align}
and using integration by parts and Young's inequality, the following estimate is  obtained:
\begin{align}\label{equ-V-dot-law}
\nonumber\dot V_1(t)\leq &-\bigg(\varrho_{1}e^{-\varrho_2}\bigg(\varrho_{2}-\frac{2\bar\sigma+\bar\omega+2\rVert\hat c\rVert_\infty+\rVert\hat \kappa\rVert_\infty}{\underline\lambda}\\
\nonumber&-\frac{\bar \mu\rVert\hat m_1\rVert_\infty}{\underline \lambda}\bigg)-\frac{2\epsilon e^{\varrho_2 }}{\underline \mu}\bigg)\rVert\hat u\rVert^2\\
\nonumber&-\bigg(\varrho_2-\frac{\varrho_1(\bar\omega+\rVert\hat \kappa\rVert_\infty)}{\underline\lambda}-\frac{\bar \mu \bar F}{\underline \mu}-\frac{4\epsilon e^{\varrho_2 }}{\underline\mu}\bigg)\rVert \hat z\rVert^2\\
&-(1-{\varrho_1} q^2-\frac{\epsilon e^{\varrho_2 }}{\underline\mu})\hat z(0, t)^{2}+\frac{\epsilon e^{\varrho_2 }}{\underline\mu}\rVert \hat v\rVert^2\nonumber\\
&+\bigg(\frac{\bar \mu \bar F}{\underline \mu}e^{2\varrho_2}+\frac{\varrho_1\bar \mu\rVert\hat m_1\rVert_\infty}{\underline \lambda}\bigg)\tilde z(0,t)^2\nonumber\\
\nonumber&-\bigg(\varrho_3\bigg(\varrho_4-\frac{2\bar \sigma}{\underline \lambda}-\frac{2\bar \theta\|\hat r_1\|_\infty}{\underline \lambda}\bigg)e^{-\varrho_4}\\
\nonumber&-\frac{\varrho_5\bar \theta(1+\|\hat r_2\|_\infty)}{\underline \mu}e^{2\varrho_4 }-\frac{\varrho_{3}\epsilon(1+\|\hat r_1\|_\infty)}{\underline\lambda}\bigg)\|\tilde w\|^2\\
\nonumber&-\bigg(\varrho_5\bigg(\varrho_{4}-\frac{\bar \theta(1+\|\hat r_2\|_\infty)}{\underline \mu}\bigg)-\frac{4\epsilon e^{\varrho_4 }}{\underline \mu}\\
&-\frac{\varrho_{3}\epsilon(1+\|\hat r_1\|_\infty)}{\underline\lambda}\bigg)\|\tilde  z\|^2-\varrho_{5}\tilde z(0,t)^2,
\end{align}
where  $F(x)\leq \bar F$ is a bounded function and 
\begin{align}
\rVert \hat c\rVert_\infty \leq&\bar\omega\rVert\hat k_1\rVert _{\infty}{\rm e}^{\rVert \hat k_1\rVert_\infty},\\
\rVert \hat \kappa\rVert_\infty \leq&\bar\omega\rVert\hat k_2\rVert _{\infty}{\rm e}^{\rVert \hat k_2\rVert_\infty}.
\end{align}
Since the inverse transformation of the approximated gain kernel  \eqref{eq:backstepping_2-inverse} allows to derive a bound of the norm of the state $v(x,t)$ in \eqref{equ-V-dot-law} with respect to the norm of the approximated target system's state $\hat u(x,t)$ and $\hat z(x,t)$. In other words,
\begin{align}\label{approx-inverse0}
\nonumber \hat v(x, t) =&\hat z(x, t)+\int_0^x \hat l_1(x, \xi)\hat u(\xi, t) d\xi \\
&+\int_0^x \hat l_2(x, \xi)\hat  z(\xi, t) d\xi,
\end{align}
where the inverse kernel $\hat l(x,\xi)$ and  its inverse $\hat k(x,\xi)$ 
satisfy  the following equation 
\begin{align}
\hat l_1(x, \xi) & =\hat k_1(x,\xi)+\int_\xi^x \hat k_2(x, s)\hat  l_1(s, \xi) ds,\label{eq:kernel-inverse-1-hat}\\
\hat l_2(x, \xi) & =\hat k_2(x,\xi)+\int_\xi^x \hat k_2(x, s)\hat  l_2(s,\xi) ds,\label{eq:kernel-inverse-2-hat}
\end{align}
and the following conservative bounds hold
\begin{align}
\rVert \hat l_1\rVert_\infty\leq \rVert \hat k_1\rVert_\infty {\rm e}^{\rVert \hat k_2\rVert_\infty},\label{equ-l1-hat-bound}\\
\rVert \hat l_2\rVert_\infty\leq \rVert \hat k_2\rVert_\infty {\rm e}^{\rVert \hat k_2\rVert_\infty}.\label{equ-l2-hat-bound}
\end{align}
Since 
\begin{align}
\rVert k_{i}-\hat k_{i}\rVert_\infty<\epsilon,
\end{align}
it follows  that 
\begin{align}
\rVert\hat k_{i}\rVert_\infty\leq\rVert k_{i}\rVert_\infty+\epsilon,
\end{align}
and using \eqref{equ-k-bouded}, we derive the following bound
\begin{align}
\rVert\hat k_{i}\rVert_\infty\leq  N_i {\rm e}^{M_i}+\epsilon.\label{equ-ki-bound}
\end{align}
Allowing for the substitution of \eqref{equ-ki-bound} into \eqref{equ-l1-hat-bound} and \eqref{equ-l2-hat-bound} results in the following inequalities  \begin{align}
\rVert \hat l_1\rVert_\infty\leq(N_1 {\rm e}^{M_1}+\epsilon) {\rm e}^{N_2 {\rm e}^{M_2}+\epsilon},\label{equ-l1-bound}\\
\rVert \hat l_2\rVert_\infty\leq(N_2 {\rm e}^{M_2}+\epsilon) {\rm e}^{N_2 {\rm e}^{M_2}+\epsilon}.\label{equ-l2-bound}
\end{align}
Similarly, based on the inverse transformations \eqref{equ-inverse-m2} and \eqref{equ-inverse-m1}, we have 
\begin{align}
\tilde w(x, t) =&\tilde u(x,t)+\int_0^x \hat r_1(x, \xi) \tilde v(\xi,t)\mathrm d\xi,\label{equ-inverse-m1-hat}\\
\tilde z(x,t)=&\tilde v(x,t)+\int_0^x\hat r_2(x,\xi)\tilde v(\xi,t)\mathrm d\xi,\label{equ-inverse-m2-hat}
\end{align}
where the inverse kernel $\hat r_1(x,\xi)$ and  the kernel $\hat r_2(x,\xi)$
satisfy   equations
\begin{align}
\hat r_1(x, \xi) & =\hat m_1(x, \xi)-\int_\xi^x\hat  m_1(x, s) \hat r_2(s, \xi) ds,\label{eq:kernel-inverse-2-r1-hat}\\
\hat r_2(x, \xi) & =-\hat m_2(x,\xi)-\int_\xi^x \hat m_2(x, s) \hat r_2(s,\xi) ds,\label{eq:kernel-inverse-2-r2-hat}
\end{align}
and the following conservative  bounds written below  hold
\begin{align}
\rVert \hat r_1\rVert_\infty \leq&\rVert\hat m_1\rVert _{\infty}{\rm e}^{\rVert \hat m_1\rVert_\infty},\label{equ-r1-hat-bound}\\
\rVert \hat r_2\rVert_\infty \leq&\rVert\hat m_2\rVert _{\infty}{\rm e}^{\rVert \hat m_2\rVert_\infty}.\label{equ-r2-hat-bound}
\end{align}
Knowing that
\begin{align}
\rVert m_{i}-\hat m_{i}\rVert_\infty<\epsilon,
\end{align}
one can deduce that
\begin{align}
\rVert\hat m_{i}\rVert_\infty\leq\rVert m_{i}\rVert_\infty+\epsilon,
\end{align}
and using \eqref{equ-k-bouded} enables one to arrive at 
\begin{align}
\rVert\hat m_{i}\rVert_\infty\leq  N_i {\rm e}^{M_i}+\epsilon.\label{equ-mi-bound}
\end{align}
Substituting  \eqref{equ-mi-bound}
into \eqref{equ-r1-hat-bound} and \eqref{equ-r2-hat-bound} gives
\begin{align}
\rVert \hat r_1\rVert_\infty \leq&(N_1 {\rm e}^{M_1}+\epsilon){\rm e}^{N_1 {\rm e}^{M_1}+\epsilon},\label{equ-r1-bound}\\
\rVert \hat r_2\rVert_\infty \leq&(N_2 {\rm e}^{M_2}+\epsilon){\rm e}^{N_2 {\rm e}^{M_2}+\epsilon}.\label{equ-r2-bound}
\end{align}
Based on \eqref{approx-inverse0},  the following relation holds
\begin{align}
\nonumber\rVert v(t)\rVert^2 =&\int_0^1\bigg(\hat z(x, t)+\int_0^x \hat l_1(x, \xi)\hat u(\xi, t) d\xi\\
\nonumber& +\int_0^x \hat l_2(x, \xi)\hat  z(\xi, t) d\xi\bigg)^2\mathrm dx\\
\leq&3\rVert \hat l_1\rVert_\infty^{2}\rVert \hat u(t)\rVert ^2+3(1+\rVert \hat l_{2}\rVert_\infty^{2})\rVert \hat z(t)\rVert ^2.\label{equ-uv}
\end{align}
Substituting  \eqref{equ-uv} into \eqref{equ-V-dot-law} yields  the following bound
\begin{align}
\nonumber\dot V_1(t)\leq&-\bigg(\varrho_{1}e^{-\varrho_2}\bigg(\varrho_{2}-\frac{2\bar\sigma+\bar\omega+2\rVert\hat c\rVert_\infty+\rVert\hat \kappa\rVert_\infty}{\underline\lambda}\\
\nonumber&-\frac{\bar \mu\rVert\hat m_1\rVert_\infty}{\underline \lambda}\bigg)-\frac{\epsilon e^{\varrho_2 }(2+3\rVert \hat l_1\rVert_\infty^{2})}{\underline\mu}\bigg)\rVert\hat u\rVert^2\\
\nonumber&-\bigg(\varrho_2-\frac{\varrho_1(\bar\omega+\rVert\hat \kappa\rVert_\infty)}{\underline\lambda}-\frac{\bar \mu \bar F}{\underline \mu}-\frac{7\epsilon e^{\varrho_2 }}{\underline\mu}\\
\nonumber&-\frac{3\epsilon e^{\varrho_2 }\rVert \hat l_{2}\rVert_\infty^{2}}{\underline\mu}\bigg)\rVert \hat z\rVert^2-(1-{\varrho_1} q^2-\frac{\epsilon e^{\varrho_2 }}{\underline\mu})\hat z(0, t)^{2}\nonumber\\
\nonumber&-\bigg(\varrho_3\bigg(\varrho_4-\frac{2\bar \sigma}{\underline \lambda}-\frac{2\bar \theta\|\hat r_1\|_\infty}{\underline \lambda}\bigg)e^{-\varrho_4}\\
\nonumber&-\frac{\varrho_5\bar \theta(1+\|\hat r_2\|_\infty)}{\underline \mu}e^{2\varrho_4 }-\frac{\varrho_{3}\epsilon(1+\|\hat r_1\|_\infty)}{\underline\lambda}\bigg)\|\tilde w\|^2\\
\nonumber&-\bigg(\varrho_5(\varrho_{4}-\frac{\bar \theta(1+\|\hat r_2\|_\infty)}{\underline \mu})-\frac{4\epsilon e^{\varrho_2 }}{\underline \mu}\\
\nonumber&-\frac{\varrho_{3}\epsilon(1+\|\hat r_1\|_\infty)}{\underline\lambda}\bigg)\|\tilde z\|^2-\bigg(\varrho_5-\frac{\bar \mu \bar F}{\underline \mu}e^{2\varrho_4}\\
&-\frac{\varrho_1\bar \mu\rVert\hat m_1\rVert_\infty}{\underline \lambda}\bigg)\tilde z(0,t)^2.\label{equ-V-dot-2}
\end{align}
Hence, selecting the parameters for the Lyapunov function $V_1$ as
\begin{align} 
&0<\varrho_1<\min\{\frac{\underline\lambda(\underline \mu \varrho_2-{\bar \mu \bar F})}{\underline \mu(\bar\omega+\rVert\hat \kappa\rVert_\infty)},\frac{1}{q^2}\},\\
&\varrho_{2} >\max\bigg\{\frac{2\bar\sigma+\bar\omega+2\rVert\hat c\rVert_\infty+\rVert\hat \kappa\rVert_\infty+\bar \mu\rVert\hat m_1\rVert_\infty}{\underline\lambda},\  \frac{\bar \mu \bar F}{\underline \mu}\bigg\},\\
 &\varrho_3>\frac{\underline \lambda \varrho_5\bar \theta(1+\|\hat r_2\|_\infty)e^{3\varrho_4 }}{\underline \mu(\varrho_4\underline \lambda-2(\bar \sigma+\bar \theta\|\hat r_1\|_\infty))},\\
&\varrho_4>\max\{\frac{\bar \theta(1+\|\hat r_2\|_\infty)}{\underline \mu},\ \frac{2\bar \sigma+2\bar \theta\|\hat r_1\|_\infty}{\underline \lambda}\},\\
&\varrho_5>\frac{\bar \mu \bar Fe^{2\varrho_4}}{\underline \mu}+\frac{\varrho_1\bar \mu\rVert\hat m_1\rVert_\infty}{\underline \lambda},
\end{align}
one can define $\epsilon^*$ as
\begin{align}
\nonumber\epsilon^*=&\min\bigg\{\frac{\underline\mu \varrho_{1}}{ e^{2\varrho_2 }(2+3\rVert \hat l_1\rVert_\infty^{2})}\bigg(\varrho_{2}-\frac{2\bar\sigma+\bar\omega+2\rVert\hat c\rVert_\infty+\rVert\hat \kappa\rVert_\infty}{\underline\lambda}\\
\nonumber&-\frac{\bar \mu\rVert\hat m_1\rVert_\infty}{\underline \lambda}\bigg),\ \bigg(\frac{\varrho_2\underline\lambda-\varrho_1(\bar\omega+\rVert\hat \kappa\rVert_\infty)}{\underline\lambda}-\frac{\bar \mu \bar F}{\underline \mu}\bigg)\\
\nonumber&\cdot\frac{\underline\mu}{e^{\varrho_2 }(7+3\rVert \hat l_{2}\rVert_\infty^{2})},\  \frac{\varrho_5\underline\lambda(\varrho_{4}\underline \mu-\bar \theta(1+\|\hat r_2\|_\infty))}{4\underline\lambda e^{\varrho_4 }+\underline \mu \varrho_{3}(1+\|\hat r_1\|_\infty)}  \\
\nonumber&\frac{\underline\lambda}{ \varrho_{3}(1+\|\hat r_1\|_\infty)}\bigg(\frac{\varrho_3(\varrho_4\underline \lambda-2(\bar \sigma+\bar \theta\|\hat r_1\|_\infty))e^{-\varrho_4}}{\underline \lambda}\\
&-\frac{\varrho_5\bar \theta(1+\|\hat r_2\|_\infty)}{\underline \mu}e^{2\varrho_4 }\bigg),\ \underline\mu(1-{\varrho_1} q^2) e^{-\varrho_2 }\bigg\},
\end{align}
such that for all  $\epsilon\in (0,\epsilon^*)$,
\begin{align}
\dot V_1(t)\leq -\vartheta_{1}V_1(t),
\end{align}
where $\vartheta_1(\epsilon)$ is defined by
\begin{align}
\nonumber \vartheta_{1}(\epsilon)=&\min\bigg\{\frac{\underline \lambda}{\varrho_1}\bigg(\varrho_{1}e^{-\varrho_2}\bigg(\varrho_{2}-\frac{2\bar\sigma+\bar\omega+2\rVert\hat c\rVert_\infty+\rVert\hat \kappa\rVert_\infty}{\underline\lambda}\\
\nonumber&-\frac{\bar \mu\rVert\hat m_1\rVert_\infty}{\underline \lambda}\bigg)-\frac{\epsilon e^{\varrho_2 }(2+3\rVert \hat l_1\rVert_\infty^{2})}{\underline\mu}\bigg)\bigg),\\
\nonumber&\frac{\underline \mu}{e^{\varrho_2}}\bigg(\varrho_2-\frac{\varrho_1(\bar\omega+\rVert\hat \kappa\rVert_\infty)}{\underline\lambda}-\frac{\bar \mu \bar F}{\underline \mu}-\frac{7\epsilon e^{\varrho_2 }}{\underline\mu}\\
\nonumber&-\frac{3\epsilon e^{\varrho_2 }\rVert \hat l_{2}\rVert_\infty^{2}}{\underline\mu}\bigg),\ \frac{\underline \lambda e^{-\varrho_4}}{\varrho_3}\bigg(\varrho_3\bigg(\varrho_4-\frac{2\bar \sigma}{\underline \lambda}-\frac{2\bar \theta\|\hat r_1\|_\infty}{\underline \lambda}\bigg)\\
\nonumber&-\frac{\varrho_5\bar \theta(1+\|\hat r_2\|_\infty)}{\underline \mu}e^{2\varrho_4 }-\frac{\varrho_{3}\epsilon(1+\|\hat r_1\|_\infty)}{\underline\lambda}\bigg),\\
\nonumber&\ \frac{\underline \mu}{\varrho_5e^{\varrho_4}}\bigg(\varrho_5\bigg(\varrho_{4}-\frac{\bar \theta(1+\|\hat r_2\|_\infty)}{\underline \mu}\bigg)-\frac{\varrho_{3}\epsilon(1+\|\hat r_1\|_\infty)}{\underline\lambda}\\
&-\frac{4\epsilon e^{\varrho_4 }}{\underline \mu}\bigg)\bigg\},\label{equ-vartheta-1}
\end{align}
which leads to the following inequality \begin{align}
V_1(t)\leq V_1(0)e^{-\vartheta_1(\epsilon) t}.
\end{align} 
From \eqref{equ-Psi}, we have 
\begin{align}
 V_1(t) \leq&\max\bigg\{\frac{\varrho_1}{\underline \lambda},\ \frac{\varrho_3}{\underline \lambda}, \ \frac{e^{\varrho_2}}{\underline \mu},\ \frac{\varrho_{5}e^{\varrho_4 }}{\underline \mu} \ \bigg\}\Psi_1(t),\\
\Psi_1(t)\leq &\frac{1}{\min\bigg\{\frac{\varrho_1e^{-\varrho_2 }}{\bar \lambda},\  \frac{\varrho_3e^{-\varrho_4}}{\bar \lambda},\ \frac{1}{\bar\mu},\ \frac{\varrho_{5}}{\bar \mu} \bigg\}}V_1(t). 
\end{align}
Therefore, the  exponential stability bound \eqref{eq-Psi-exp-bound1} holds, and \begin{align}
\nonumber \vartheta_2=&\min\bigg\{\frac{\varrho_1e^{-\varrho_2 }}{\bar \lambda},\  \frac{\varrho_3e^{-\varrho_4}}{\bar \lambda},\ \frac{1}{\bar\mu},\ \frac{\varrho_{5}}{\bar \mu} \bigg\}\\
&\cdot\max\bigg\{\frac{\varrho_1}{\underline \lambda},\ \frac{\varrho_3}{\underline \lambda}, \ \frac{e^{\varrho_2}}{\underline \mu},\ \frac{\varrho_{5}e^{\varrho_4 }}{\underline \mu} \ \bigg\}.
\end{align}
  \hfill
\end{proof}

The following proposition is introduced to state the stability equivalence between the target system and the original closed-loop system.  There is a norm equivalence between Transformations \eqref{eq:kernel-inverse-beta-hat}, \eqref{eq:kernel-inverse-m2-hat}, along with their inverse \eqref{approx-inverse0}, \eqref{equ-inverse-m1-hat} and  \eqref{equ-inverse-m2-hat} help to state the following norm-equivalence properties.

\medskip\begin{propo}\label{propo-norm-equal-0}
{\bf\em [norm equivalence with DeepONet kernels]}
Consider the closed-loop system including the plant \eqref{eq:sys_u}--\eqref{eq:sys_BC_2} with  observer system  \eqref{eq:sys_u-obs}--\eqref{eq:sys_BC_2-obs} and the observer-based controller \eqref{eq:U-hat}.
 There exists $\epsilon^*>0$ such that for all $\epsilon\in (0,\epsilon^*),$ the following estimates hold  between this closed-loop system  and the target system \eqref{equ-w0-hat}--\eqref{equ-z1-hat}, 
\begin{align}
\Psi_1(t)\leq S_1(\epsilon)\Phi_1(t),\quad  \Phi_1(t) \leq S_2(\epsilon)\Psi_1(t),\quad 
\end{align}
where 
\begin{align}
\Phi_1(t)=\rVert u(t)\rVert^2+\rVert   v(t)\rVert^2+\rVert\hat u(t)\rVert^2+\rVert\hat v(t)\rVert^2,\label{equ-Phi-original}
\end{align}
$\Psi_1(t)$ is defined in \eqref{equ-Psi} and  the positive constants as
\begin{align}
\nonumber S_1(\epsilon)=&20+8(N_1 {\rm e}^{M_1}+\epsilon){\rm e}^{N_1 {\rm e}^{M_1}+\epsilon}+8(N_2 {\rm e}^{M_2}+\epsilon)\\
&\cdot {\rm e}^{N_2 {\rm e}^{M_2}+\epsilon}+3(N_1 {\rm e}^{M_1}+N_2 {\rm e}^{M_2}+2\epsilon),\label{equ-S-1-0}\\
\nonumber S_2(\epsilon)=&20+9(N_1 {\rm e}^{M_1}+N_2 {\rm e}^{M_2}+2\epsilon)\ {\rm e}^{ N_2 {\rm e}^{M_2}+\epsilon}+4N_1 {\rm e}^{M_1}\\
&+4N_2 {\rm e}^{M_2}+8\epsilon.\label{equ-S-2-0}
\end{align}
\end{propo}
\begin{proof}\rm
From \eqref{equ-trans-wu}--\eqref{eq:kernel-inverse-m2-hat}, we have 
\begin{align}
\nonumber\Psi_1(t)
\nonumber=&\rVert \hat u(t)\rVert^2+\int_0^1\bigg(\hat v (x,t)-\int_0^x \hat k_1(x, \xi) \hat u(\xi,t) d\xi\\
\nonumber&-\int_0^x \hat k_2(x, \xi) \hat v(\xi,t) d\xi\bigg)^2\mathrm dx\\
\nonumber&+\int_0^1\bigg(\tilde u(x,t)+\int_0^x \hat r_1(x, \xi) \tilde v(\xi,t)\mathrm d\xi\bigg)^2\mathrm dx\\\nonumber&+\int_0^1\bigg(\tilde v(x,t)+\int_0^x \hat r_2(x, \xi) \tilde v(\xi,t)\mathrm d\xi\bigg)^2\mathrm dx\\
\nonumber\leq&\rVert \hat u(t)\rVert^2+3\int_0^1\bigg(\int_0^x \hat k_1(x, \xi) \hat u(\xi,t) d\xi\bigg)^2\mathrm dx\\
\nonumber&+3\rVert \hat v(t)\rVert^2+3\int_0^1\bigg(\int_0^x \hat k_2(x, \xi) \hat v(\xi,t) d\xi\bigg)^2\mathrm dx\\
\nonumber&+2\rVert \tilde u(t)\rVert^2+2\int_0^1\bigg(\int_0^x \hat r_1(x, \xi) \tilde v(\xi,t)\mathrm d\xi\bigg)^2\mathrm dx\\\nonumber&+2\rVert \tilde v(t)\rVert^2+2\int_0^1\bigg(\int_0^x \hat r_2(x, \xi) \tilde v(\xi,t)\mathrm d\xi\bigg)^2\mathrm dx\\
\nonumber\leq&(1+3\rVert\hat k_{1}\rVert_\infty^2)\rVert \hat u(t)\rVert^2+3(1+\rVert\hat k_{2}\rVert_\infty^2)\rVert \hat v(t)\rVert^2\\
&+2\rVert \tilde u(t)\rVert^2+2(1+\rVert\hat r_{1}\rVert_\infty^2+\rVert\hat r_{2}\rVert_\infty^2)\rVert \tilde v(t)\rVert^2.
\end{align}
Since $\tilde u=u-\hat u$ and $\tilde v=v-\hat v$, we have 
\begin{align}
\nonumber\Psi_1(t)&\leq(1+3\rVert\hat k_{1}\rVert_\infty^2)\rVert \hat u(t)\rVert^2+3(1+\rVert\hat k_{2}\rVert_\infty^2)\rVert \hat v(t)\rVert^2\\
\nonumber&+2\rVert u(t)-\hat u(t)\rVert^2+2(1+\rVert\hat r_{1}\rVert_\infty^2+\rVert\hat r_{2}\rVert_\infty^2)\\
\nonumber&\cdot\rVert v(t)-\hat v(t)\rVert^2\\
\nonumber\leq&4\rVert u(t)\rVert^2+(5+3\rVert\hat k_{1}\rVert_\infty^2)\rVert \hat u(t)\rVert^2+4(1+\rVert\hat r_{1}\rVert_\infty^2\\
\nonumber&+\rVert\hat r_{2}\rVert_\infty^2)\rVert v\rVert^2+(7+4\rVert\hat r_{1}\rVert_\infty^2+4\rVert\hat r_{2}\rVert_\infty^2+3\rVert\hat k_{2}\rVert_\infty^2)\\
\nonumber&\cdot\rVert\hat v(t)\rVert^2\\
\leq&(20+3\rVert\hat k_{1}\rVert_\infty^2+3\rVert\hat k_{2}\rVert_\infty^2+8\rVert\hat r_{1}\rVert_\infty^2+8\rVert\hat r_{2}\rVert_\infty^2)\Phi_1(t).\label{equ-Psi1-1}
\end{align}
Submiting \eqref{equ-ki-bound}, \eqref{equ-r1-bound}, and  \eqref{equ-r2-bound} into \eqref{equ-Psi1-1}, it arrivals 
\begin{align}
\nonumber &\Psi_1(t)\leq\left(20+8(N_1 {\rm e}^{M_1}+\epsilon){\rm e}^{N_1 {\rm e}^{M_1}+\epsilon}+8(N_2 {\rm e}^{M_2}+\epsilon)\right.\\
&~~~~~~~~~~\left.\cdot{\rm e}^{N_2 {\rm e}^{M_2}+\epsilon}+3(N_1 {\rm e}^{M_1}+N_2 {\rm e}^{M_2}+2\epsilon)\right)\Phi_1(t).
\end{align}
Similarly, from \eqref{eq:kernel-inverse-m1-hat}, \eqref{eq:kernel-inverse-m2-hat}, and \eqref{approx-inverse0}, we obtain
\begin{align}
\nonumber \Phi_1(t)
\nonumber=&\rVert \hat u(t)+\tilde u(t)\rVert^2+\rVert \hat v(t)+ \tilde v(t)\rVert^2+\rVert\hat u(t)\rVert^2+\rVert \hat v(t)\rVert^2\\
\nonumber\leq&3\rVert \hat u(t)\rVert^2+3\int_0^1\bigg(\hat z(x, t)+\int_0^x \hat l_1(x, \xi)\hat u(\xi, t) d\xi \\
\nonumber&+\int_0^x \hat l_2(x, \xi)\hat z(\xi, t) d\xi\bigg)^2\mathrm dx+2\int_0^1\bigg(\tilde w(x,t)\\
\nonumber&+\int_0^x \hat m_1(x, \xi) \tilde z(\xi,t) d\xi\bigg)^2\mathrm dx+2\int_0^1\bigg(\tilde z (x,t)\\\nonumber&+\int_0^x \hat m_2(x, \xi) \tilde  z(\xi,t) d\xi\bigg)^2\mathrm dx\\
\nonumber\leq&(3+9 \|\hat l_1\|^{2}_\infty)\rVert \hat u(t)\rVert^2+9(1+ \|\hat l_2\|^{2}_\infty)\rVert \hat z (t)\rVert^2 \\
\nonumber&+4\rVert\tilde w(t)\rVert^2+4(1+ \|\hat m_1\|^{2}_\infty+ \|\hat m_2\|^{2}_\infty)\rVert\tilde z (t)\rVert^2\\
\nonumber\leq&(20+9(N_1 {\rm e}^{M_1}+N_2 {\rm e}^{M_2}+2\epsilon)\ {\rm e}^{ N_2 {\rm e}^{M_2}+\epsilon}\\
&+4N_1 {\rm e}^{M_1} +4N_2 {\rm e}^{M_2}+8\epsilon)\Psi_1(t),
\end{align}
which completes the proof.
\mbox{}\hfill
\end{proof}

After establishing the norm-equivalence  in Proposition \ref{propo-norm-equal-0}, the main result immediately follows in Theorem \ref{theo2}. \\
\medskip\begin{theorem}\label{theo2}{\bf\em [main result---stabilization with DeepONet]} 
Consider the closed-loop system consisting of the plant \eqref{eq:sys_u}--\eqref{eq:sys_BC_2}  together with  the observer \eqref{eq:sys_u-obs}--\eqref{eq:sys_BC_2-obs} and the control law \eqref{eq:U-hat}.  Assuming that functions  $\lambda,\ \mu  \in C^1([0, 1])$ have  Lipschitz derivatives, $\sigma,\ \omega,\ \theta  \in C^0([0, 1])$, $
q\in \mathbbm R$, and satisfy $\rVert \lambda \rVert_\infty \leq B_\lambda,\ \rVert \mu \rVert_\infty  \leq  B_\mu,\ \rVert \sigma \rVert_\infty  \leq B_\sigma,\ \rVert \omega \rVert_\infty  \leq B_\omega,\ \rVert \theta \rVert_\infty  \leq B_\theta,\ \rVert \lambda' \rVert_\infty  \leq B_{\lambda'},\ \rVert \mu' \rVert_\infty  \leq  B_{\mu'}$, where  $B_\lambda,\ B_\mu,\ B_\sigma,\ B_\omega,\ B_\theta,\ B_{\lambda'},\ B_{\mu'}>0$ are arbitrarily large constants, there exists a sufficiently small $\epsilon^*(B_\lambda, B_\mu,  B_\sigma,  B_\omega,  B_\theta,  B_{\lambda'},  B_{\mu'})>0$ such that all gain in the feedback law \eqref{eq:U-hat} and the observer system  \eqref{eq:sys_u-obs}--\eqref{eq:sys_BC_2-obs}  with the  neural operator $\hat {\mathcal{M}}(\lambda, \mu, \sigma, \omega, \theta, q)$ of approximation accuracy $\epsilon\in(0,\epsilon^*)$ in relation to the exact backstepping kernels $k_i(x,\xi)$, and $m_i(x,\xi)$, $\ i=1,2$ that ensures the following exponential stability bound 
\begin{align}\label{stabilitybound}
\Phi_1(t)\leq \Phi_1(0){S_1(\epsilon)}{S_2(\epsilon)}\vartheta_2 {\rm e}^{-\vartheta_1(\epsilon) t}\,,
\quad\forall t\geq 0\,,
\end{align}
where,   $\vartheta_1,\ \vartheta_2>0$ are positive constants,   $\Phi_1(t)$, $S_1(\epsilon)$ and $S_2(\epsilon)$ are defined in \eqref{equ-Phi-original}--\eqref{equ-S-2-0}, respectively.
\end{theorem}
\begin{remark}
  \rm  The product $S_1 S_2 $ is the portion of the overshoot which depends on $\epsilon$ and this dependence is clearly increasing, based on \eqref{equ-S-1-0} and \eqref{equ-S-2-0}. It makes sense that poor approximation increases the overshoot estimate. The definition of the decay rate $\vartheta_1,$ as given \eqref{equ-vartheta-1}, shows a decreasing dependence on $\epsilon$, meaning that a poor approximation reduces the decay rate estimate.
\end{remark}


\section{A Fully Learned  Output Feedback Law via DeepONet approximation}\label{Sec5}
\subsection{Summary of the design procedure }

In this section, we present a DeepONet approximation design that enables one to achieve learning of the output-feedback boundary control signal and provide proof-equipped stability guarantees. Exploiting the kernel functions approximation obtained in Section \ref{Stab-Deep}, we design a DeepONet that takes as entries the five plant's parameters $\lambda(x),\ \mu(x),\ \sigma(x),\ \omega(x),\ \theta(x)$ and $q$, as well as the estimates generated by the state observer, namely, $\hat u(x,t)$, $\hat v(x,t).$ The Learning network is built for complete learning of control law restated below for the sake of clarity
\begin{align}
\hat U(t)  =&\int_0^1 \hat k_1(1, \xi) \hat u(\xi, t) d\xi +\int_0^1 \hat k_2(1, \xi)\hat v(\xi, t) d\xi.\label{equ-U-law} 
\end{align}
The structure of the DeepONet-assisted closed-loop system is depicted in  Fig. \ref{Learning-control-law}. The expansion of the mapping $\mathcal{K}$ from a larger space $\mathcal{U}$ to the scalar value of the control input comes at the price of substantial amount of training and learning effort. Furthermore, our result only ensures semi-global practical exponential stability (SG-PES) because as opposed to the approach presented in Section \ref{Stab-Deep}, which only contains multiplicative error, the mapping $\hat U(t)$ as reflected in \eqref{equ-U-law}, involves an additive intermediate linear layer that supplements additive error into the NO approximation process.  We proceed with the three following steps (see Fig. \ref{Learning-control-law}):
\begin{figure}[t]
\centering
\includegraphics[width=0.405\textwidth]{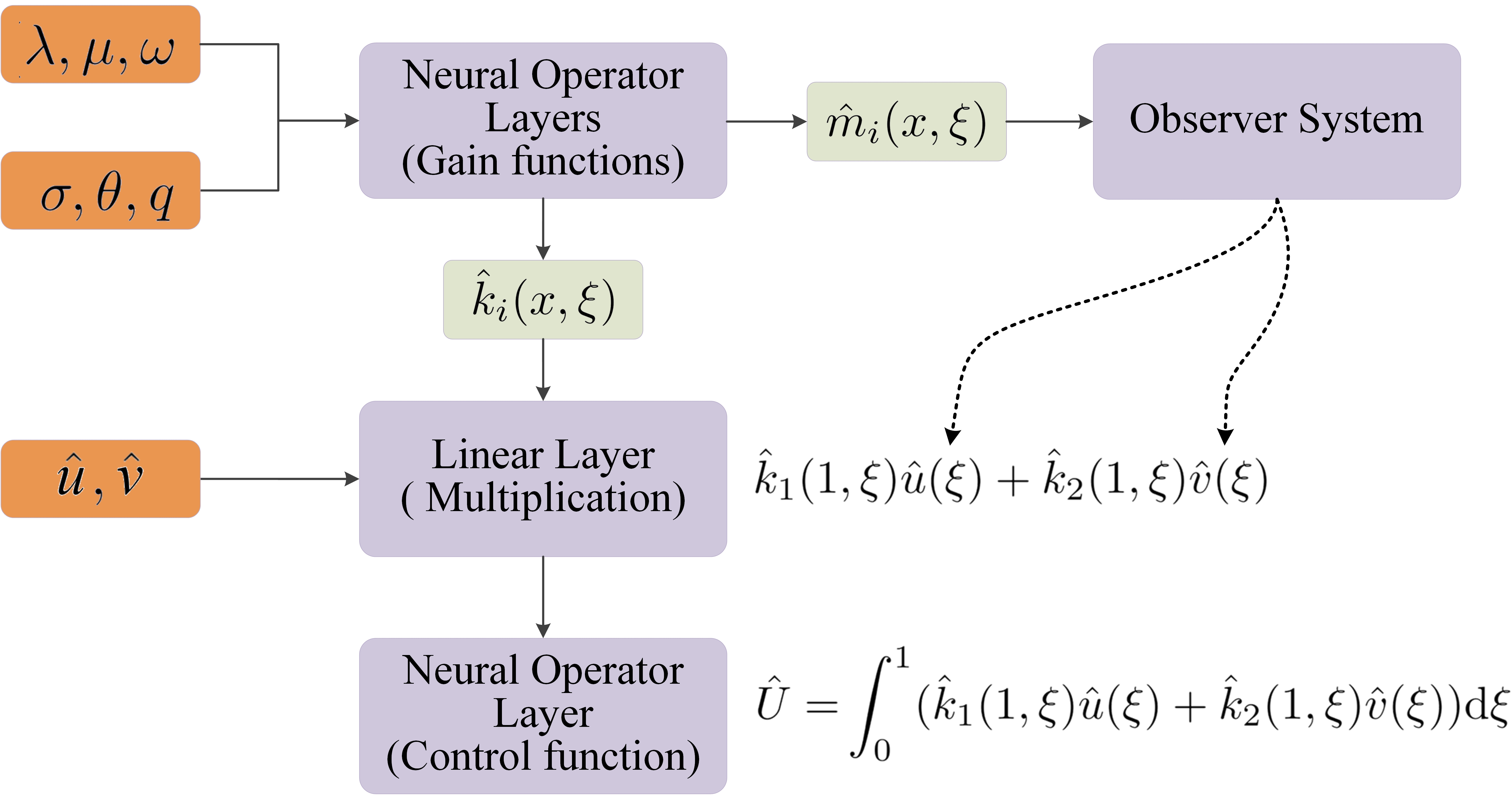}
\vspace{-0.15cm}
\caption{The learning architecture of the observer-based control law in three steps.} \label{Learning-control-law}
\end{figure}

\begin{itemize}   
\item  \textbf{Step 1.} The functions  $\lambda(x),\ \mu(x),\ \sigma(x),\ \omega(x),\ \theta(x),\ q$ remain the inputs of through neural operators $\mathcal{K}$ introduced in Section 4 and generates the NO approximated kernel functions  $\hat k_i(x,\xi)$ and $\hat m_i(x,\xi)$, $i=1,2$. 

\item \textbf{Step 2.}  A linear layer is utilized to multiply the estimated kernel functions $\hat k_i(x,\xi),\ i=1,2$ with the observer's estimates $\hat u$ and $\hat v$.

\item \textbf{Step 3.} A new neural operator ${\mathcal U}:(\lambda, \mu, \sigma, \omega, \theta, q,\hat u,\hat v)\mapsto U$, $(C^1[0,1])^2\times (C^0[0,1])^3\times  \mathbbm R\times(C^0[0,1])^2 \mapsto \mathbb R$, where $U$ is defined in \eqref{eq:U-hat}, is learned to implement the nonlinear integral operation, resulting in the final observer-based control law $\hat U$ given by \eqref{equ-U-law}. This mapping is developed using the DeepONet approximation accuracy theorem recently introduced in  \cite{krstic2023neural} for a reaction-diffusion PDE.

\end{itemize}

The expansion of the mapping $\mathcal{K}$ defined in \textbf{Step 1} from larger space $\mathcal{U}$ to the scalar value of the control input $\hat U (t)$ comes at the price of a substantial amount of training and learning effort. 

Let's denote $\hat{\mathcal{U}}$ the NO approximation of the  output-feedback operator  ${\mathcal U}:(\lambda, \mu, \sigma, \omega, \theta, q,\hat u,\hat v)\mapsto U$, $(C^1[0,1])^2\times (C^0[0,1])^3\times  \mathbbm R\times(C^0[0,1])^2 \mapsto \mathbb R$, and recall the operator ${\mathcal{\hat M}}$ given in Theorem \ref{thm-karniadakis-bkst},  the following theorem holds.

\medskip\begin{theorem}\label{theo-observer-law}
Let $B_\lambda,\ B_\mu,\ B_\sigma,\ B_\omega,\ B_\theta,\ B_{\lambda'},\ B_{\mu'}>0$ be arbitrarily large and $\epsilon>0$, there exists  neural operators ${\mathcal{\hat M}}$ and $\hat{\mathcal  U}$ such that
\begin{align}
\nonumber&|{\mathcal U(\lambda, \mu, \sigma, \omega, \theta, q,\hat u, \hat v)}-\hat{\mathcal U}(\lambda, \mu, \sigma, \omega, \theta, q,\hat  u, \hat v)|\\
&+|{\mathcal{M}}(\lambda, \mu,  \sigma, \omega, \theta,
q)(x,\xi)-\hat{\mathcal{M}}(\lambda, \mu,  \sigma, \omega, \theta,
q)(x,\xi)|<\epsilon,
\end{align}
holds for all Lipschitz $\lambda,\ \mu,\ \sigma,\ \omega,\ \theta,\ \hat u,\ \hat v$ with the properties that $\rVert \lambda \rVert_\infty  \leq B_\lambda,\ \rVert \mu \rVert_\infty  \leq  B_\mu,\ \rVert \sigma \rVert_\infty  \leq B_\sigma,\ \rVert \omega \rVert_\infty  \leq B_\omega,\ \rVert \theta \rVert_\infty  \leq B_\theta,\ \rVert \lambda' \rVert _\infty \leq B_{\lambda'},\ \rVert \mu' \rVert_\infty  \leq  B_{\mu'}$,  $\rVert\hat  u(t)\rVert_\infty\leq B_{u}$, $\rVert\hat  v(t)\rVert_\infty\leq B_{v}$,  namely, there exists  neural operators $\hat {\mathcal{K}}$ such that 
\begin{align}
\nonumber&|\tilde k_1|+|\tilde k_2|+|\tilde c|+|\tilde \kappa|+|(\lambda(x) + \mu(x))\tilde k_1(x, x)|\\
\nonumber&+|\lambda(0)q\tilde k_1(x, 0)-\mu(0)\tilde k_2(x, 0)| \\
\nonumber&+|-\mu(x) \partial_x\tilde  k_1+\lambda(\xi) \partial_\xi \tilde k_1+\lambda^\prime(\xi)\tilde  k_1+\sigma(\xi)\tilde  k_1+\theta(\xi) \tilde k_2|\\
\nonumber&+|-\mu(x) \partial_x\tilde  k_2-\mu(\xi) \partial_\xi \tilde k_2-\mu^\prime(\xi)\tilde  k_2 +\omega(\xi)\tilde  k_1|\\
\nonumber&\textcolor{black}{+|\lambda(x) \partial_x \tilde m_1-\mu(\xi) \partial_\xi \tilde m_1+\mu^\prime(\xi)\tilde m_1-\sigma(\xi) \tilde m_1}-\omega(x) \tilde m_2|\\
\nonumber&\textcolor{black}{+|\mu(x) \partial_x \tilde m_2+\mu(\xi) \partial_\xi \tilde m_2+\mu^\prime(\xi) \tilde m_2}\textcolor{black}{+\theta(\xi) \tilde m_1|}\\
&+|\tilde{\mathcal U}(\lambda, \mu, \sigma, \omega, \theta, q, \hat u, \hat v)|<\epsilon.
\end{align}
\end{theorem}

\begin{proof}\rm
    The continuity of the operator $\mathcal{M}$ follows directly from Lemma \ref{lem1} and that of the operator $\mathcal{U}$ can be established following  \cite[Lem. 4]{bhan2023neural}. The final result is then obtained by invoking \cite[Thm. 2.1]{lu2021advectionDeepONet}.
\mbox{}\hfill
\end{proof}

Theorem \ref{theo-observer-law} is used to prove the stability of \eqref{eq:sys_u}--\eqref{eq:sys_BC_2} combined with the observer system \eqref{eq:sys_u-obs}--\eqref{eq:sys_BC_2-obs} when the approximated output feedback control 
 law  \eqref{equ-U-law} learned through DeepOnet is assigned. By collecting training input-output data generated by a range of constants  $q$ and a family of spatially varying parameter functions  $\lambda(x),\ \mu(x),\ \sigma(x),\ \omega(x),\ \theta(x)$ as well as the observer states $\hat u(x)$ and $\hat v(x)$.

\subsection{Stabilization under output feedback control law generated via DeepONet}


Recalling the NO approximation $\hat{\mathcal{U}},$ the approximated control \eqref{equ-U-law}  can be expressed as    $\hat U=\hat{\mathcal{U}}(\lambda, \mu, \sigma, \omega, \theta, q,\hat u,\hat v).$
Applying the certainty equivalence principle, the approximated backstepping transformations 
\eqref{eq:backstepping_2}, \eqref{eq:kernel-inverse-m1} and  \eqref{eq:kernel-inverse-m2} fed by the $\hat k_i$ and $\hat m_i$, $i=1,2$, are defined as \eqref{eq:kernel-inverse-beta-hat}--\eqref{eq:kernel-inverse-m2-hat}, respectively.
The inverse transformations of
\eqref{eq:kernel-inverse-beta-hat}--\eqref{eq:kernel-inverse-m2-hat} are defined in \eqref{approx-inverse0}, \eqref{equ-inverse-m1-hat} and \eqref{equ-inverse-m2-hat}, respectively.

Using the backstepping transformation \eqref{eq:kernel-inverse-beta-hat}, the observer \eqref{eq:sys_BC_2-obs}--\eqref{eq:sys_u-obs} translates into the following target system
\begin{align}
\nonumber&\partial_t\hat u(x,t)=-\lambda(x) \partial_x \hat u(x,t)+\sigma(x)\hat u(x,t)+\omega(x)\hat z(x,t)\\
\nonumber&~~~~~~+\int_0^x c(x,\xi)\hat u(\xi,t)d\xi +\int_0^x  \kappa(x,\xi)\hat z(\xi,t)d\xi\\
&~~~~~~~~+\hat m_1(x,0)\mu(0)\tilde z(0,t),\label{equ-w0-hat-law}\\
 &\partial_t\hat z(x,t)=\mu(x) \partial_x \hat z(x,t)+F(x)\mu(0)\tilde z(0,t),\\
&\hat u(0,t)=q\hat z(0,t),\\
&\hat z(1,t)=\tilde U(t),\label{equ-z1-law}
\end{align}
where $\hat \kappa(x,\xi),\ \hat c(x,\xi)$ and $F(x)$ are defined in \eqref{hat-kap}, \eqref{hat-c}, and \eqref{equ-F}, respectively. The approximation error terms,  $\delta_i,\ i=1,2,3,4$  are given in \eqref{delta-1}--\eqref{delta-4}, 
and $\tilde U(t)=U(t)-\hat U(t)$.
We recall that  $U(t)$, the approximated control law  \eqref{equ-U-law}, is obtained from an approximation of the gain kernel when functions parameters $\lambda(x),\ \mu(x),\ \sigma(x),\ \omega(x),\ \theta(x)$ vary whereas the full approximation of the feedback law, namely, $\hat U(t)$, requires input-output data of the observer states, namely, $\hat u$ and $\hat v,$  provided some $L^2$ initial data $(u_0(x), v_0(x), \hat u_0(x), \hat v_0(x)).$ It is worth recalling that the estimated state trajectories are derived from a dataset collected at the sensing point $v(0, t)$.


Using  \eqref{eq:kernel-inverse-m1-hat}--\eqref{eq:kernel-inverse-m2-hat}, the error system \eqref{eq:sys_u-err}--\eqref{eq:sys_BC_2-err} maps into the following set of PDEs
\begin{align}
\nonumber&\partial_t \tilde w(x,t)=-\lambda(x) \partial_x \tilde w(x,t)+\sigma(x)\tilde w(x,t)\\
\nonumber&~~~~~~+\int_0^x \hat g(x,\xi) \tilde w(\xi,t)\mathrm d\xi+\int_0^x \delta_5(x,\xi) \tilde z(\xi,t)\mathrm d\xi \\
&~~~~~~ +\int_0^x\int_\xi^x \hat r_1(x,s)\delta_6(s,\xi)\mathrm ds\tilde z(\xi,t)\mathrm d\xi,\label{equ-alpha-error}\\
 &\partial_t\tilde z(x,t)=\mu(x) \partial_x \tilde z(x,t)+\theta(x)\tilde w(x,t)\nonumber\\
\nonumber&~~~~~~+\int_0^x \hat h(x,\xi) \tilde w(\xi,t)\mathrm d\xi+\int_0^x \delta_6(x,\xi) \tilde z(\xi,t)\mathrm d\xi \\
&~~~~~~ +\int_0^x\int_\xi^x \hat r_2(x,s)\delta_6(s,\xi)\mathrm ds\tilde z(\xi,t)\mathrm d\xi,\\
&\tilde w(0,t)=0,\\
&\tilde z(1,t)=\tilde U(t).\label{equ-z1-hat-law}
\end{align}
where $\hat g(x,\xi)$, $\hat h(x,\xi)$, $\delta_5(x,\xi)$ and $\delta_6(x,\xi)$ are defined in \eqref{hat-g}, \eqref{hat-h}, \eqref{delta-5} and \eqref{delta-6}, respectively.

Our first result for the coupled target system \eqref{equ-w0-hat-law}--\eqref{equ-z1-law}, \eqref{equ-alpha-error}--\eqref{equ-z1-hat-law} is its semi-global practical exponential stability in the backstepping-transformed variables under the DeepONet-approximated kernels.

\medskip\begin{propo}\label{prop-esti-law}
Consider the cascaded target system \eqref{equ-w0-hat-law}--\eqref{equ-z1-law}, \eqref{equ-alpha-error}--\eqref{equ-z1-hat-law}, there exists $\varepsilon^*>0$ such that for all $\varepsilon\in (0,\varepsilon^*)$, and  the following holds,
\begin{align}
\label{eq-Psi-exp-bound2}
\Psi_2(t)\leq \Psi_2(0)\vartheta_4 {\rm e}^{-\vartheta_3(\epsilon) t}+\vartheta_5\epsilon^2, \quad \forall t\geq 0,
\end{align}
where $\vartheta_i>0$, $i=3,4,5$, and
\begin{align}
\Psi_2(t)=\rVert \hat u(t)\rVert^2+\rVert \hat z (t)\rVert^2+\rVert\tilde w(t)\rVert^2+\rVert\tilde  z(t)\rVert^2.\label{equ-Psi-law}
\end{align}
\end{propo}


\begin{proof}\rm
Define the Lyapunov function 
\begin{align}
\nonumber &V_2(t) =\int_0^1 \frac{\iota _1e^{-\iota _2 x}}{\lambda(x)}\hat u (x, t)^2   dx +\int_0^1  \frac{e^{\iota _2x}}{\mu(x)}\hat z(x, t)^2 dx\\
  &+\int_0^1 \frac{\iota _3e^{-\iota _4x}}{\lambda(x)}\tilde  w(x, t)^2   dx +\int_0^1  \frac{\iota _{5}e^{\iota _4 x}}{\mu(x)}\tilde z(x, t)^2 dx,\label{eq:Lyap_V1-law}
\end{align}
where $\iota _i>0,\ i=1,2,..,5$ are constants to be selected. 

Computing the   time derivative of \eqref{eq:Lyap_V1-law} along \eqref{equ-w0-hat-law}--\eqref{equ-z1-law}, \eqref{equ-alpha-error}--\eqref{equ-z1-hat-law} as
\begin{align}
\nonumber\dot V_2(t)
\nonumber=&2\int_0^1 \frac{\iota _1e^{-\iota _2x}}{\lambda(x)} \hat u(x, t)\bigg(-\lambda(x) \partial_x \hat u(x,t)\\
\nonumber&+\sigma(x)\hat u(x,t)+\omega(x)\hat z(x,t)\\
\nonumber&+\int_0^x   c(x,\xi)\hat u(\xi,t)d\xi +\int_0^x   \kappa(x,\xi)\hat z(\xi,t)d\xi\\
&+\hat m_1(x,0)\mu(0)\tilde z(0,t)\bigg)+2\int_0^1  \frac{e^{\iota _2x}}{\mu(x)}\hat z(x, t)\nonumber\\
\nonumber&\cdot\bigg(\mu(x) \partial_x \hat z(x,t)+F(x)\mu(0)\tilde z(0,t)\bigg) dx\\
\nonumber&+2\int_0^1 \frac{\iota _3e^{-\iota _4x}}{\lambda(x)} \tilde w(x, t)\bigg(-\lambda(x) \partial_x \tilde w(x,t)\\
\nonumber&+\sigma(x)\tilde w(x,t)+\int_0^x \hat g(x,\xi) \tilde w(\xi,t)d\xi \\
\nonumber&+\int_0^x \delta_5(x,\xi) \tilde z(\xi,t)\mathrm d\xi \\
&+\int_0^x\int_\xi^x \hat r_1(x,s)\delta_6(s,\xi)\mathrm ds\tilde z(\xi,t)\mathrm d\xi\bigg)dx\nonumber\\
\nonumber&+2\int_0^1  \frac{\iota_5e^{\iota_4x}}{\mu(x)}\tilde z(x, t)\bigg(\mu(x) \partial_x \tilde z(x,t)+\theta(x)\tilde w(x,t)\nonumber\\
\nonumber&+\int_0^x \hat h(x,\xi) \tilde w(\xi,t)\mathrm d\xi+\int_0^x \delta_6(x,\xi) \tilde z(\xi,t)\mathrm d\xi \\
& +\int_0^x\int_\xi^x \hat r_2(x,s)\delta_6(s,\xi)\mathrm ds\tilde z(\xi,t)\mathrm d\xi\bigg) dx,
\end{align}
and using integration by parts and Young's inequality, the following estimate is  derived
\begin{align}
\nonumber\dot V_2(t)\leq &-\iota _{1}e^{-\iota _2}\bigg(\iota _{2}-\frac{2\bar\sigma+\bar\omega+2\rVert c\rVert_\infty+\rVert\kappa\rVert_\infty}{\underline\lambda}\\
\nonumber&-\frac{\bar \mu\rVert\hat m_1\rVert_\infty}{\underline \lambda}\bigg)\rVert\hat u\rVert^2-\bigg(\iota _2-\frac{\iota _1(\bar\omega+\rVert\kappa\rVert_\infty)}{\underline\lambda}\\
&-\frac{\bar \mu \bar F}{\underline \mu}\bigg)\rVert \hat z\rVert^2-(1-{\iota _1} q^2)\hat z(0, t)^{2}+e^{\iota_2}\hat z(1, t)^{2}\nonumber\\
&+\bigg(\frac{\bar \mu \bar F}{\underline \mu}e^{2\iota _2}+\frac{\iota _1\bar \mu\rVert\hat m_1\rVert_\infty}{\underline \lambda}\bigg)\tilde z(0,t)^2\nonumber\\
\nonumber&-\bigg(\iota _3\bigg(\iota _4-\frac{2\bar \sigma}{\underline \lambda}-\frac{2\bar \theta\|\hat r_1\|_\infty}{\underline \lambda}\bigg)e^{-\iota _4}\\
\nonumber&-\frac{\iota _5\bar \theta(1+\|\hat r_2\|_\infty)}{\underline \mu}e^{2\iota _4 }-\frac{\iota _{1}\epsilon(1+\|\hat r_1\|_\infty)}{\underline\lambda}\bigg)\|\tilde w\|^2\\
\nonumber&-\iota _{5}\tilde z(0,t)^2-\bigg(\iota _5(\iota _{4}-\frac{\bar \theta(1+\|\hat r_2\|_\infty)}{\underline \mu})-\frac{4\epsilon e^{\iota _4 }}{\underline \mu}\\
&-\frac{\iota _{3}\epsilon(1+\|\hat r_1\|_\infty)}{\underline\lambda}\bigg)\|\tilde z\|^2+\iota _{5}e^{\iota _4 }\tilde z(1,t)^2,
\end{align}
where
\begin{align}
\rVert  c\rVert_\infty \leq&\bar\omega\rVert  k_1\rVert _{\infty}{\rm e}^{\rVert  k_1\rVert_\infty},\label{equ-c-bound}\\
\rVert  \kappa\rVert_\infty \leq&\bar\omega\rVert  k_2\rVert _{\infty}{\rm e}^{\rVert  k_2\rVert_\infty}.\label{equ-kappa-bound}
\end{align}
Substituting  \eqref{equ-k-bouded} into \eqref{equ-c-bound} and \eqref{equ-kappa-bound}, we  get \begin{align}
\rVert c\rVert_\infty\leq\bar\omega N_1 {\rm e}^{M_1+N_1 {\rm e}^{M_1}},\label{equ-l1-bound-law}\\
\rVert \kappa\rVert_\infty\leq\bar\omega N_2 {\rm e}^{M_2+N_2 {\rm e}^{M_2}}.\label{equ-l2-bound-law}
\end{align}
Since we have designed the neural operator $\mathcal{U}$ to learn the mapping $(\lambda,\mu,\delta,\omega,\theta,u,v)\mapsto U$ such that $\hat U(t)={\mathcal{U}}(\lambda,\mu,\delta,\omega,\theta,q,\hat u(t),\hat v(t))$, we can get $|U(t)-\hat U(t)|\leq \epsilon$. Hence, the estimate below can be derived
\begin{align}
\nonumber\dot V_2(t)=&-\iota _{1}e^{-\iota _2}\bigg(\iota _{2}-\frac{2\bar\sigma+\bar\omega+2\rVert c\rVert_\infty+\rVert\kappa\rVert_\infty}{\underline\lambda}\\
\nonumber&-\frac{\bar \mu\rVert\hat m_1\rVert_\infty}{\underline \lambda}\bigg)\rVert\hat u\rVert^2-\bigg(\iota _2-\frac{\iota _1(\bar\omega+\rVert\kappa\rVert_\infty)}{\underline\lambda}\\
&-\frac{\bar \mu \bar F}{\underline \mu}\bigg)\rVert \hat z\rVert^2-(1-{\iota _1} q^2)\hat z(0, t)^{2}\nonumber\\
&-\bigg(\iota _{5}-\frac{\bar \mu \bar F}{\underline \mu}e^{2\iota _2}-\frac{\iota _1\bar \mu\rVert\hat m_1\rVert_\infty}{\underline \lambda}\bigg)\tilde z(0,t)^2\nonumber\\
\nonumber&-\bigg(\iota _3\bigg(\iota _4-\frac{2\bar \sigma}{\underline \lambda}-\frac{2\bar \theta\|\hat r_1\|_\infty}{\underline \lambda}\bigg)e^{-\iota _4}\\
\nonumber&-\frac{\iota _5\bar \theta(1+\|\hat r_2\|_\infty)}{\underline \mu}e^{2\iota _4 }-\frac{\iota _{1}\epsilon(1+\|\hat r_1\|_\infty)}{\underline\lambda}\bigg)\|\tilde w\|^2\\
\nonumber&-\bigg(\iota _5(\iota _{4}-\frac{\bar \theta(1+\|\hat r_2\|_\infty)}{\underline \mu})-\frac{4\epsilon e^{\iota _4 }}{\underline \mu}\\
&-\frac{\iota _{3}\epsilon(1+\|\hat r_1\|_\infty)}{\underline\lambda}\bigg)\|\tilde z\|^2+(e^{\iota_2}+\iota _{5}e^{\iota _4 })\epsilon^2.
\end{align}
For the time-derivative of the Lyapunov function $V_2(t)$. We can set that
\begin{align} 
&0<\iota_1<\frac{1}{q^2},\\
&\iota_{2} >\max\bigg\{\frac{2\bar\sigma+\bar\omega+2\rVert c\rVert_\infty+\rVert\kappa\rVert_\infty+\bar \mu\rVert\hat m_1\rVert_\infty}{\underline\lambda},\nonumber\\
&~~~~~\frac{\iota _1(\bar\omega+\rVert\kappa\rVert_\infty)}{\underline\lambda}+  \frac{\bar \mu \bar F}{\underline \mu}\bigg\},\\
 &\iota_3>\frac{\underline \lambda \iota_5\bar \theta(1+\|\hat r_2\|_\infty)e^{3\iota_4 }}{\underline \mu(\iota_4\underline \lambda-2(\bar \sigma+\bar \theta\|\hat r_1\|_\infty))},\\
&\iota_4>\max\{\frac{\bar \theta(1+\|\hat r_2\|_\infty)}{\underline \mu},\ \frac{2\bar \sigma+2\bar \theta\|\hat r_1\|_\infty}{\underline \lambda}\},\\
&\iota_5>\frac{\bar \mu \bar Fe^{2\iota_2}}{\underline \mu}+\frac{\iota_1\bar \mu\rVert\hat m_1\rVert_\infty}{\underline \lambda},
\end{align}
one can define $\epsilon^*$ as
\begin{align}
\nonumber\epsilon^*&=\min\bigg\{\frac{\underline\lambda}{ \iota_{1}(1+\|\hat r_1\|_\infty)}\bigg(\frac{\iota_3(\iota_4\underline \lambda-2(\bar \sigma+\bar \theta\|\hat r_1\|_\infty))e^{-\iota_4}}{\underline \lambda}\\
&-\frac{\iota_5\bar \theta(1+\|\hat r_2\|_\infty)}{\underline \mu}e^{2\iota_4 }\bigg),\ \frac{\iota_5\underline\lambda(\iota_{4}\underline \mu-\bar \theta(1+\|\hat r_2\|_\infty))}{4\underline\lambda e^{\iota_2 }+\underline \mu \iota_{3}(1+\|\hat r_1\|_\infty)}\bigg\},
\end{align}
such that for all  $\epsilon\in (0,\epsilon^*)$,
\begin{align}
&\dot V_2(t)\leq -\vartheta_{3}(\epsilon)V_2(t)+(e^{\iota_2}+\iota _{5}e^{\iota _4 })\epsilon^2,
\end{align}
where $\vartheta_3(\epsilon)$ is defined by
\begin{align}
\nonumber \vartheta_{3}(\epsilon)&=\min\bigg\{{\underline \lambda}e^{-\iota_2}\bigg(\iota_{2}-\frac{2\bar\sigma+\bar\omega+2\rVert  c\rVert_\infty+\rVert \kappa\rVert_\infty}{\underline\lambda}\\
\nonumber&-\frac{\bar \mu\rVert\hat m_1\rVert_\infty}{\underline \lambda}\bigg),\frac{\underline \mu}{e^{\iota_2}}\bigg(\iota_2-\frac{\iota_1(\bar\omega+\rVert\kappa\rVert_\infty)}{\underline\lambda}-\frac{\bar \mu \bar F}{\underline \mu}\bigg),\\
\nonumber&\frac{\underline \lambda e^{-\iota_4}}{\iota_3}\bigg(\iota_3\bigg(\iota_4-\frac{2\bar \sigma}{\underline \lambda}-\frac{2\bar \theta\|\hat r_1\|_\infty}{\underline \lambda}\bigg)-\frac{\iota_{3}\epsilon(1+\|\hat r_1\|_\infty)}{\underline\lambda}\\
\nonumber&-\frac{\iota_5\bar \theta(1+\|\hat r_2\|_\infty)}{\underline \mu}e^{2\iota_4 }\bigg),\ \frac{\underline \mu}{\iota_5e^{\iota_4}}\bigg(\iota_5\bigg(\iota_{4}-\frac{\bar \theta(1+\|\hat r_2\|_\infty)}{\underline \mu}\bigg)\\
&-\frac{4\epsilon e^{\iota_4 }}{\underline \mu}-\frac{\iota_{3}\epsilon(1+\|\hat r_1\|_\infty)}{\underline\lambda}\bigg)\bigg\}.
\end{align}
Thus, we have the following inequality 
\begin{align}
V_2(t)\leq V_2(0)e^{-\vartheta_3(\epsilon) t}+\frac{e^{\iota_2}+\iota _{5}e^{\iota _4 }}{\vartheta_3}\epsilon^{2}.\end{align} 
From \eqref{equ-Psi-law}, we have 
\begin{align}
 V_2(t) \leq&\max\bigg\{\frac{\iota_1}{\underline \lambda},\ \frac{\iota_3}{\underline \lambda}, \ \frac{e^{\iota_2}}{\underline \mu},\ \frac{\iota_{5}e^{\iota_4 }}{\underline \mu} \ \bigg\}\Psi_2(t),\\
\Psi_2(t)\leq &\frac{1}{\min\bigg\{\frac{\iota_1e^{-\iota_2 }}{\bar \lambda},\  \frac{\iota_3e^{-\iota_4}}{\bar \lambda},\ \frac{1}{\bar\mu},\ \frac{\iota_{5}}{\bar \mu} \bigg\}}V_2(t). 
\end{align}
Therefore, the  exponential stability bound \eqref{eq-Psi-exp-bound2} holds, and 
\begin{align}
\vartheta_4(\epsilon)=&\frac{e^{\iota_2}+\iota _{5}e^{\iota _4 }}{\vartheta_3(\epsilon)},\\
\nonumber \vartheta_5=&\min\bigg\{\frac{\iota_1e^{-\iota_2 }}{\bar \lambda},\  \frac{\iota_3e^{-\iota_4}}{\bar \lambda},\ \frac{1}{\bar\mu},\ \frac{\iota_{5}}{\bar \mu} \bigg\}\\
&\cdot\max\bigg\{\frac{\iota_1}{\underline \lambda},\ \frac{\iota_3}{\underline \lambda}, \ \frac{e^{\iota_2}}{\underline \mu},\ \frac{\iota_{5}e^{\iota_4 }}{\underline \mu} \ \bigg\},
\end{align}
we have completed this proof.
  \hfill
\end{proof}


To translate the stability of the cascaded target system into that of the original closed-loop system,  we consider transformations \eqref{eq:kernel-inverse-beta-hat}--\eqref{eq:kernel-inverse-m2-hat}, along with inverse transformations \eqref{approx-inverse0}, \eqref{equ-inverse-m1-hat} and \eqref{equ-inverse-m2-hat}, and provide the following proposition.

\medskip\begin{propo}\label{propo-norm-equal}
{\bf\em [norm equivalence with DeepONet kernels]}
Consider the closed-loop system including the plant \eqref{eq:sys_u}--\eqref{eq:sys_BC_2} with  observer system  \eqref{eq:sys_u-obs}--\eqref{eq:sys_BC_2-obs} and the observer-based controller \eqref{equ-U-law}.
 There exists $\epsilon^*>0$ such that for all $\epsilon\in (0,\epsilon^*),$ the following estimates hold  between this closed-loop system  and the cascaded target system  \eqref{equ-w0-hat-law}--\eqref{equ-z1-law}, \eqref{equ-alpha-error}--\eqref{equ-z1-hat-law}, 
\begin{align}
\Psi_2(t)\leq S_1(\epsilon)\Phi_2(t),\quad  \Phi_2(t) \leq S_2(\epsilon)\Psi_2(t),\quad 
\end{align}
where 
\begin{align}
\Phi_2(t)=\rVert u(t)\rVert^2+\rVert v(t)\rVert^2+\rVert \hat u(t)\rVert^2+\rVert \hat v(t)\rVert^2,
\end{align}
and  the positive constants are defined in \eqref{equ-S-1-0} and \eqref{equ-S-2-0}, respectively.

\end{propo}
\begin{proof}\rm
The proof of Proposition \ref{propo-norm-equal} is similar to that of Proposition \ref{propo-norm-equal-0}. From \eqref{eq:kernel-inverse-beta-hat}--\eqref{eq:kernel-inverse-m2-hat}, we have 
\begin{align}
\nonumber\Psi_2(t)
\nonumber=&\rVert \hat u(t)\rVert^2+\int_0^1\bigg(\hat v (x,t)-\int_0^x \hat k_1(x, \xi) \hat u(\xi,t) d\xi\\
\nonumber&-\int_0^x \hat k_2(x, \xi) \hat v(\xi,t) d\xi\bigg)^2\mathrm dx\\
\nonumber&+\int_0^1\bigg(\tilde u(x,t)+\int_0^x \hat r_1(x, \xi) \tilde v(\xi,t)\mathrm d\xi\bigg)^2\mathrm dx\\
\nonumber&+\int_0^1\bigg(\tilde v(x,t)+\int_0^x \hat r_2(x, \xi) \tilde v(\xi,t)\mathrm d\xi\bigg)^2\mathrm dx\\
\nonumber\leq&\rVert \hat u(t)\rVert^2+3\int_0^1\bigg(\int_0^x \hat k_1(x, \xi) \hat u(\xi,t) d\xi\bigg)^2\mathrm dx\\
\nonumber&+3\rVert \hat v(t)\rVert^2+3\int_0^1\bigg(\int_0^x \hat k_2(x, \xi) \hat v(\xi,t) d\xi\bigg)^2\mathrm dx\\
\nonumber&+2\rVert \tilde u(t)\rVert^2+2\int_0^1\bigg(\int_0^x \hat r_1(x, \xi) \tilde v(\xi,t)\mathrm d\xi\bigg)^2\mathrm dx\\\nonumber&+2\rVert \tilde v(t)\rVert^2+2\int_0^1\bigg(\int_0^x \hat r_2(x, \xi) \tilde v(\xi,t)\mathrm d\xi\bigg)^2\mathrm dx\\
\nonumber\leq&(1+3\rVert \hat k_{1}\rVert_\infty^2)\rVert \hat u(t)\rVert^2+3(1+\rVert \hat k_{2}\rVert_\infty^2)\rVert \hat v(t)\rVert^2\\
\nonumber&+2\rVert u(t)-\hat u(t)\rVert^2+2(1+\rVert\hat r_{1}\rVert_\infty^2+\rVert\hat r_{2}\rVert_\infty^2)\\
\nonumber&\cdot\rVert v(t)-\hat v(t)\rVert^2\\
\leq&(20+3\rVert \hat k_{1}\rVert_\infty^2+3\rVert \hat k_{2}\rVert_\infty^2+8\rVert\hat r_{1}\rVert_\infty^2+8\rVert\hat r_{2}\rVert_\infty^2)\Phi_2(t).\label{equ-Psi2-1}
\end{align}
Submiting \eqref{equ-k-bouded}, \eqref{equ-r1-bound}, and  \eqref{equ-r2-bound} into \eqref{equ-Psi2-1}, it arrivals 
\begin{align}
\nonumber\Psi_2(t)\leq&(20+8(N_1 {\rm e}^{M_1}+\epsilon){\rm e}^{N_1 {\rm e}^{M_1}+\epsilon}+8(N_2 {\rm e}^{M_2}+\epsilon)\\
&\cdot{\rm e}^{N_2 {\rm e}^{M_2}+\epsilon}+3(N_1 {\rm e}^{M_1}+N_2 {\rm e}^{M_2}+2\epsilon))\Phi_2(t).
\end{align}

Similarly, from \eqref{approx-inverse0}, \eqref{equ-inverse-m1-hat} and \eqref{equ-inverse-m2-hat}, we obtain
\begin{align}
\nonumber \Phi_2(t)
\nonumber\leq&3\rVert \hat u(t)\rVert^2+3\rVert \hat v(t)\rVert^2+2\rVert\tilde u(t)\rVert^2+2\rVert \tilde v(t)\rVert^2\\
\nonumber\leq&3\rVert \hat u(t)\rVert^2+3\int_0^1\bigg(\hat z(x, t)+\int_0^x\hat l_1(x, \xi)\hat u(\xi, t) d\xi \\
\nonumber&+\int_0^x\hat l_2(x, \xi)\hat  z(\xi, t) d\xi\bigg)^2\mathrm dx+2\int_0^1\bigg(\tilde w(x,t)\\
\nonumber&+\int_0^x \hat m_1(x, \xi) \tilde z(\xi,t) d\xi\bigg)^2\mathrm dx+2\int_0^1\bigg(\tilde z (x,t)\\\nonumber&+\int_0^x \hat m_2(x, \xi) \tilde z(\xi,t) d\xi\bigg)^2\mathrm dx\\
\nonumber\leq&(3+9 \| \hat l_1\|^{2}_\infty)\rVert \hat u(t)\rVert^2+9(1+ \|\hat l_2\|^{2}_\infty)\rVert \hat z(t)\rVert^2 \\
\nonumber&+4\rVert\tilde w(t)\rVert^2+4(1+ \|\hat m_1\|^{2}_\infty+ \|\hat m_2\|^{2}_\infty)\rVert\tilde  z(t)\rVert^2\\
\nonumber\leq&(20+9(N_1 {\rm e}^{M_1}+N_2 {\rm e}^{M_2}+2\epsilon)\ {\rm e}^{ N_2 {\rm e}^{M_2}+\epsilon}\\
&+4N_1 {\rm e}^{M_1}+4N_2 {\rm e}^{M_2}+8\epsilon)\Psi_2(t).
\end{align}
Thus, we have completed the proof.
\mbox{}\hfill
\end{proof}

With the help of  Proposition  \ref{prop-esti-law} and  \ref{propo-norm-equal}
state we state following theorem.


\begin{theorem}\label{theo3}
 For any $\epsilon<\epsilon^*$  where 
\begin{align}
\epsilon^*:=\frac{\sqrt{(B_u^{2}+B_v^{2}+B_{\hat u}^{2}+B_{\hat v}^{2})}}{\sqrt{{S_2(\epsilon)}\vartheta_5}}>0,
\end{align}
and $\rVert u(0)\rVert ^{2}+\rVert v(0)\rVert ^{2}+\rVert \hat u(0)\rVert ^{2}+\rVert\hat v(0)\rVert ^{2}\leq\zeta$, where
\begin{align}
\zeta:=\frac{S_1(\epsilon)}{S_2(\epsilon)\vartheta_4(\epsilon)}\bigg((B_u^{2}+B_v^2+B_{\hat u}^{2}+B_{\hat v}^{2})-S_2{\vartheta_5}\epsilon^2\bigg)>0,\label{equ-zeta}
\end{align}
the closed-loop system consisting of  the NO approximation of the PDE feedback law \eqref{equ-U-law} and the plant  \eqref{eq:sys_u}--\eqref{eq:sys_BC_2}  and observer system  \eqref{eq:sys_u-obs}--\eqref{eq:sys_BC_2-obs} satisfy the semi-global practical exponential stability estimate,
\begin{align}
\Phi_2(t)\leq&\frac{S_2(\epsilon)}{S_1(\epsilon)}\vartheta_4(\epsilon) e^{-\vartheta_3(\epsilon)  t}\Phi_2(0)+S_2
\vartheta_5
\epsilon^2,  \quad\forall t\geq 0. \label{equ-result}
\end{align}
\end{theorem}

The estimate given by  \eqref{equ-result} is semi-global by permitting the radius $\zeta$ of the initial condition ball in the $L^2[0, 1]$ space to expand as the values of $B_u$, $B_v$, $B_{\hat u}$, and $B_{\hat v}$ increase.  Furthermore,  the size of the training set as well as the number of nodes of the neural network are increasing functions of  $B_u$, $B_v$, $B_{\hat u}$, and $B_{\hat v}$. Even though the stability is semi-global, the region of attraction $\zeta$ defined in \eqref{equ-zeta}, remains considerably smaller than the magnitude of the samples associated with $B_u$, $B_v$, $B_{\hat u}$, and $B_{\hat v}$ in the training set. Moreover, from \eqref{equ-result}, as $t\to\infty$, 
the residual value 
$\Phi_2(t) \leq S_2\vartheta_5\epsilon^2$, 
can be reduced to an arbitrarily small level by decreasing parameter $\epsilon$, and concurrently, by increasing both the training set size and the number of neural network nodes in accordance with the reduction of $\epsilon$. 

\section{Simulation}\label{Simulation}
Our simulation is performed considering a $2\times 2$ linear hyperbolic system with $\lambda(x) = \Gamma x+1$, $\mu(x)=e^{\Gamma x}+2$, $\delta(x)=\Gamma (x+1)$, $\theta(x)=\Gamma (x + 1)$, $\omega(x)=\Gamma(\cosh(x) + 1)$, $q=\Gamma/3$, parameterized by $\Gamma=\{2;5\}$.
Under initial conditions  $u_0(x) =1$, $v_0(x)= \sin(x)$,  the open-loop system of the plant is unstable as shown in Figure \ref{open-loop-u}. 
\begin{figure}
\centering
\includegraphics[width=0.21\textwidth]{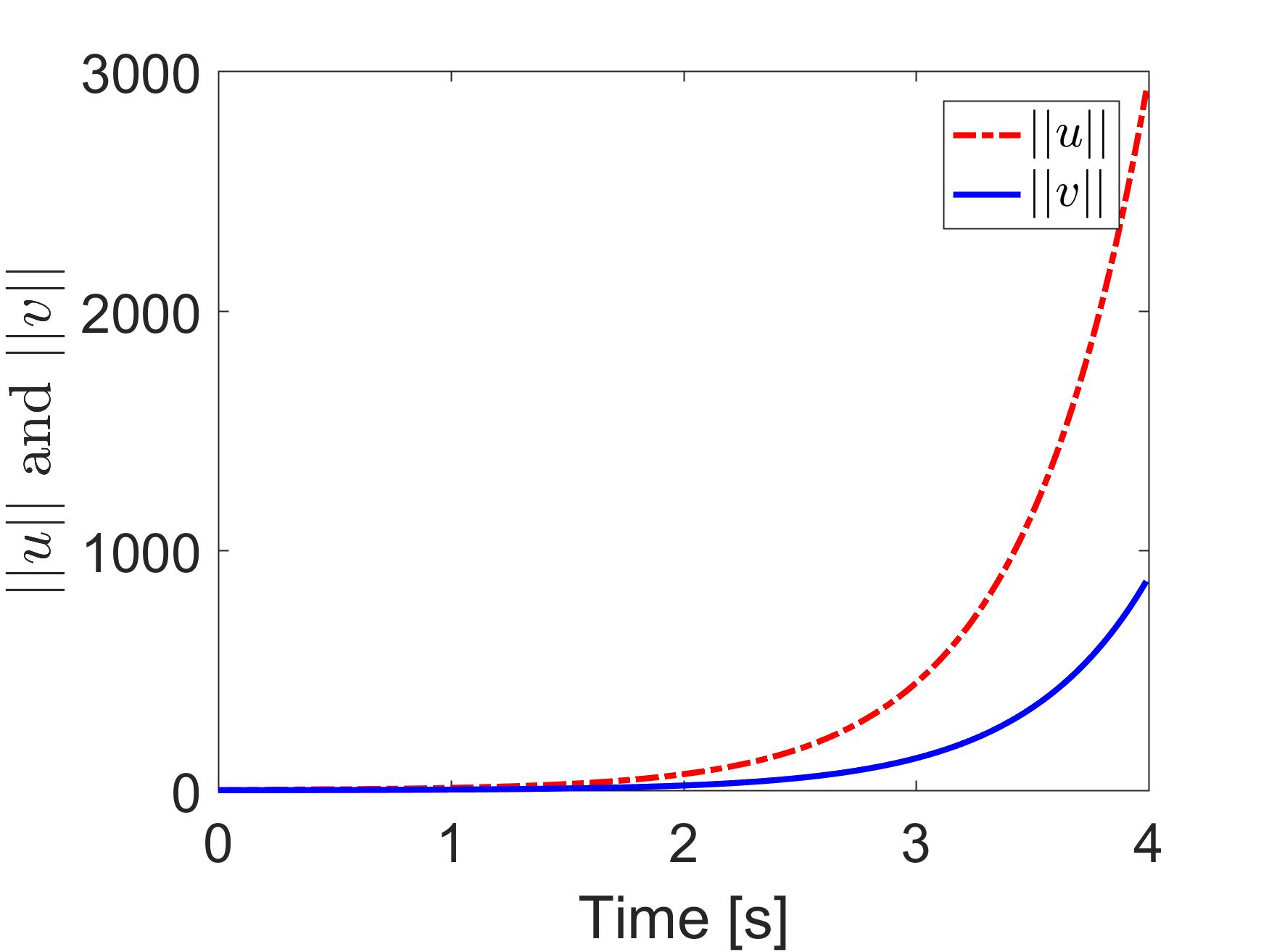}
\includegraphics[width=0.21\textwidth]{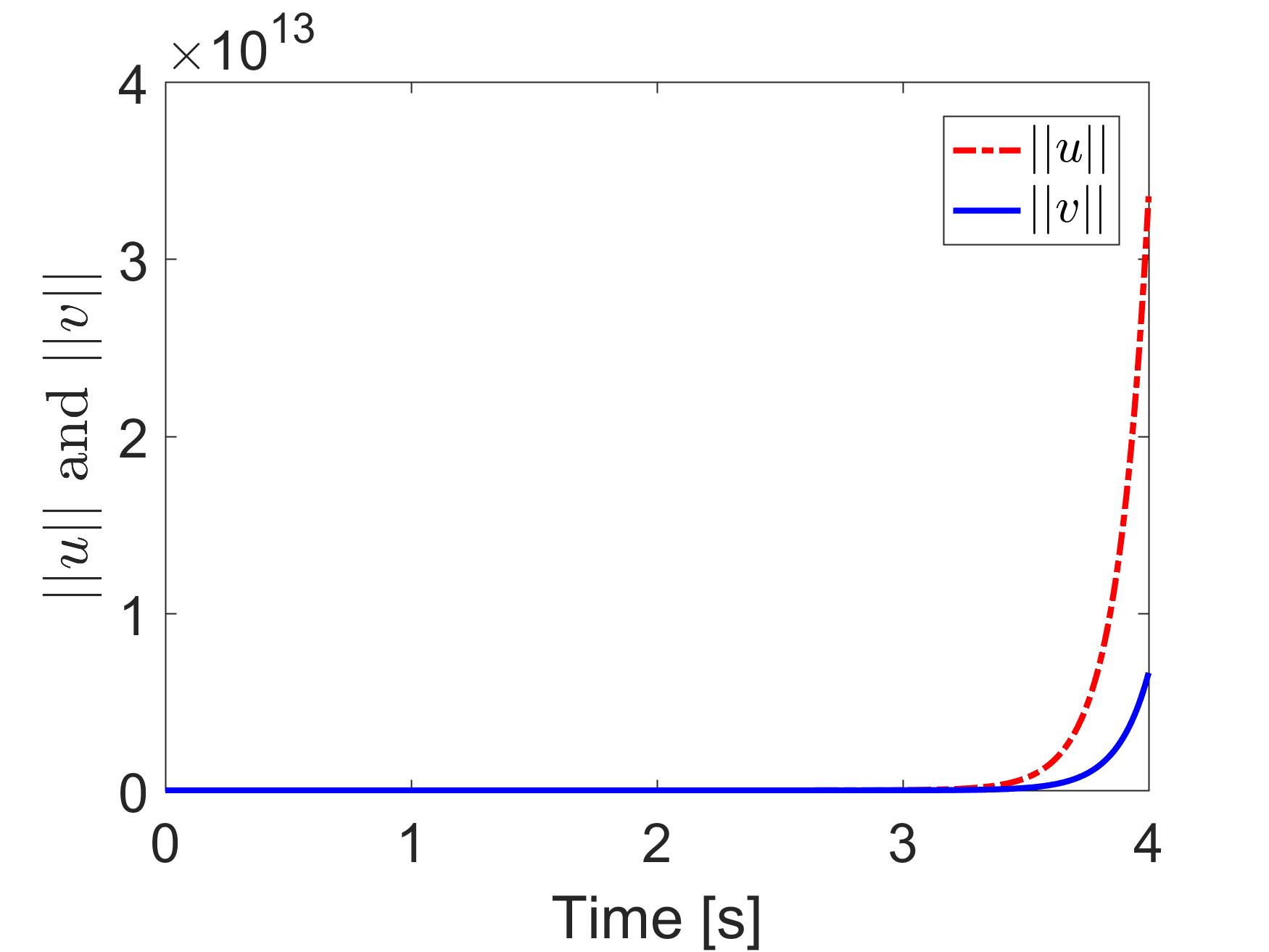}
\vspace{-0.15cm}
\caption{Instability of the uncontrolled plant of $u(x,t)$ and $v(x,t)$ for given  coefficients $\lambda(x) = \Gamma x+1$, $\mu(x)=e^{\Gamma x}+2$, $\delta(x)=\Gamma (x+1)$, $\theta(x)=\Gamma (x + 1)$, $\omega(x)=\Gamma(\cosh(x) + 1)$, $q=\Gamma/3$, $\Gamma=2,\ 5$.} \label{open-loop-u}
\end{figure}
By iterating the functions $\lambda(x)$, $\mu(x)$, $\delta(x)$, $\theta(x)$, and $\omega(x)$ along the y-axis to generate a two-dimensional (2D) input for the $\mathcal{K}$ network, the DeepONet is developed without modifying the grid structure. Similarly, the constant $q$ is iterated along both  $x$ and $y$ coordinates to generate additional 2D inputs for the $\mathcal{K}$ network. In summary, this methodology results in six distinct 2D inputs for the network. Our approach capitalizes on this 2D structure by integrating a Convolutional Neural Network (CNN) into the branch network of the DeepONet. Exploiting a  2000 samples dataset, the model demonstrating the highest accuracy in data point classification is identified.
Analytical and learned DeepONet kernels, namely $k_1$, $k_2$, $m_1$, and $m_2$, are depicted in Figures \ref{figure-k1}, \ref{figure-k2}, \ref{figure-m1}, and \ref{figure-m2}. These figures illustrate the kernels' behavior for two distinct values of $\Gamma$, specifically 2 and 5. During the training phase (refer to Fig. \ref{Train-loss}), the relative $L^2$ errors for kernels $k_1$, $k_2$, $m_1$, and $m_2$ were recorded as $4.90\times 10^{-5}$, $3.48\times 10^{-5}$, $6.69\times 10^{-5}$, and $2.61\times 10^{-5}$, respectively. The corresponding testing errors were $5.32\times 10^{-5}$, $3.89\times 10^{-5}$, $7.34\times 10^{-5}$, and $2.62\times 10^{-5}$. 

Furthermore, we simulate the closed-loop system comprising the NO approximation of the PDE feedback law \eqref{equ-U-law}, the plant \eqref{eq:sys_u}--\eqref{eq:sys_BC_2}, and the observer system \eqref{eq:sys_u-obs}--\eqref{eq:sys_BC_1-obs}. Our control law is derived using a pre-designed learning network for the gain kernels, rather than directly from the inputs $\lambda(x), \mu(x), \sigma(x), \omega(x), \theta(x), q, \hat u(x,t)$, and $\hat v(x,t)$. These inputs are processed by neural operators from Section 4 to approximate kernel functions $\hat k_i(x,\xi)$ and $\hat m_i(x,\xi)$, $i=1,2$. These approximations are then linearly combined with observer estimates $\hat u$ and $\hat v$. The final step uses a DeepONet layer to learn the mapping $(\lambda, \mu, \sigma, \omega, \theta, q, \hat u, \hat v) \to \hat U$. With 2000 samples, the model achieves an $L^2$ error of $5.46 \times 10^{-8}$ and a testing error of $5.97 \times 10^{-8}$, as shown in Figure \ref{figure-U}. These simulations corroborate the closed-loop stability under output-feedback control law. That is shown in Figures \ref{figure-u} and \ref{figure-erroru}. Moreover, Figure \ref{figure-observer-error} displays the observer error in closed-loop solutions employing observer kernels $\hat m_1(x,\xi)$ and $\hat m_2(x,\xi)$, along with the control law $\hat U(t)$.
\begin{figure}[t]
\centering
\includegraphics[width=0.405\textwidth]{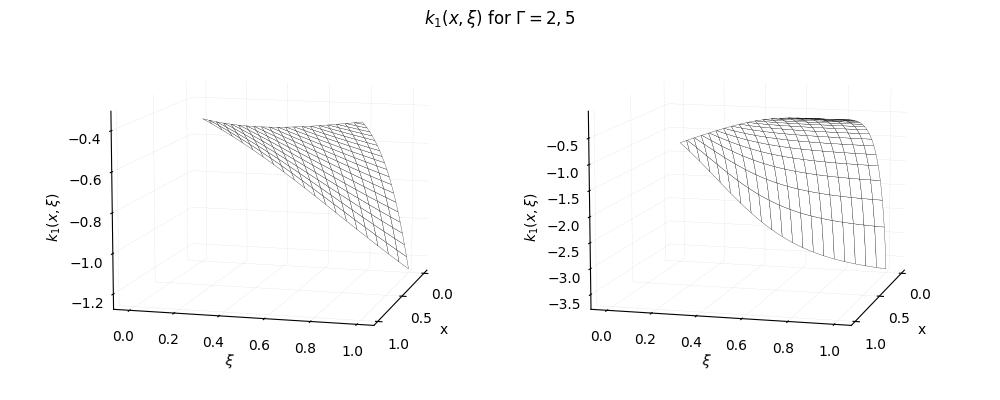}
\includegraphics[width=0.405\textwidth]{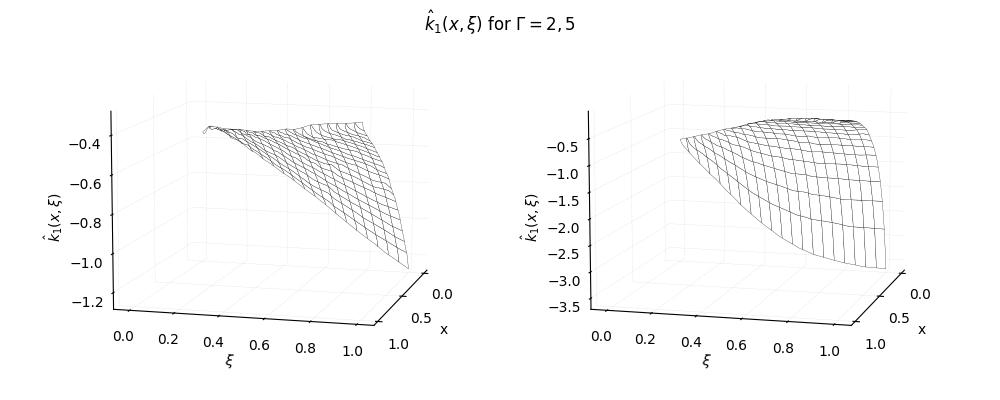}
\includegraphics[width=0.405\textwidth]{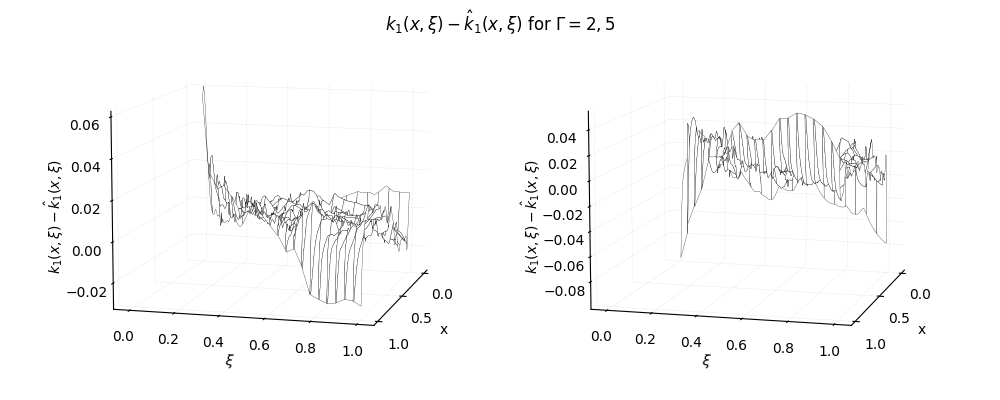}
\vspace{-0.15cm}
\caption{The kernel of $k_1(x,\xi)$, $\hat k_1(x,\xi)$ and $k_1(x,\xi)-\hat k_1(x,\xi)$.} \label{figure-k1}
\end{figure}
\begin{figure}[t]
\centering
\includegraphics[width=0.405\textwidth]{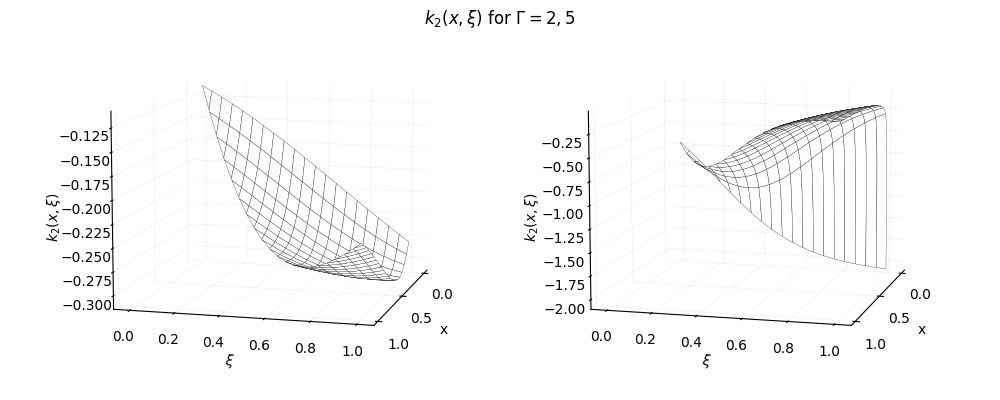}
\includegraphics[width=0.405\textwidth]{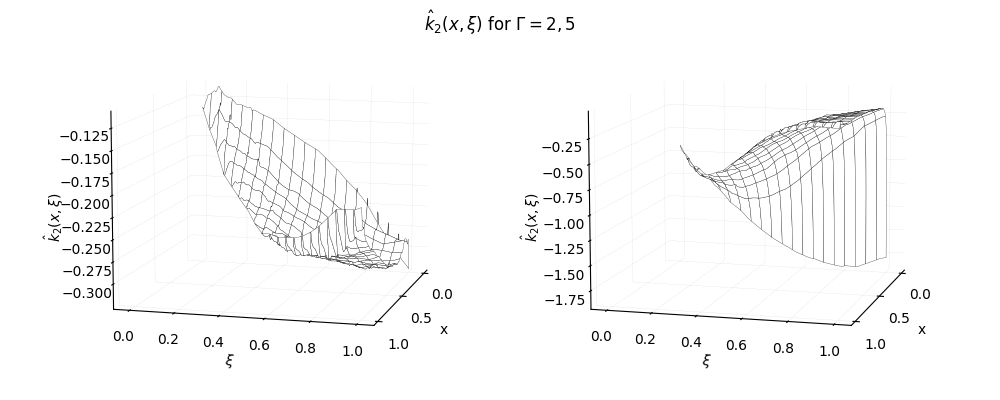}
\includegraphics[width=0.405\textwidth]{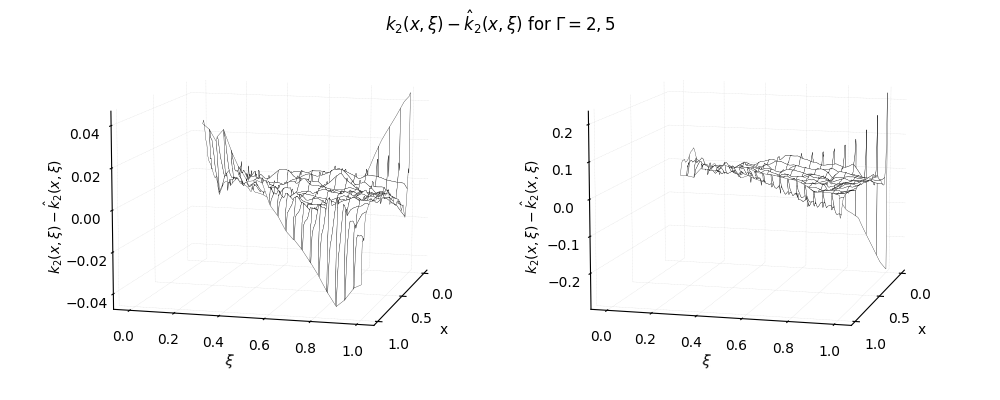}
\vspace{-0.15cm}
\caption{The kernel of $k_2(x,\xi)$, $\hat k_2(x,\xi)$ and $k_2(x,\xi)-\hat k_2(x,\xi)$.}  \label{figure-k2}
\end{figure}
\begin{figure}[t]
\centering
\includegraphics[width=0.405\textwidth]{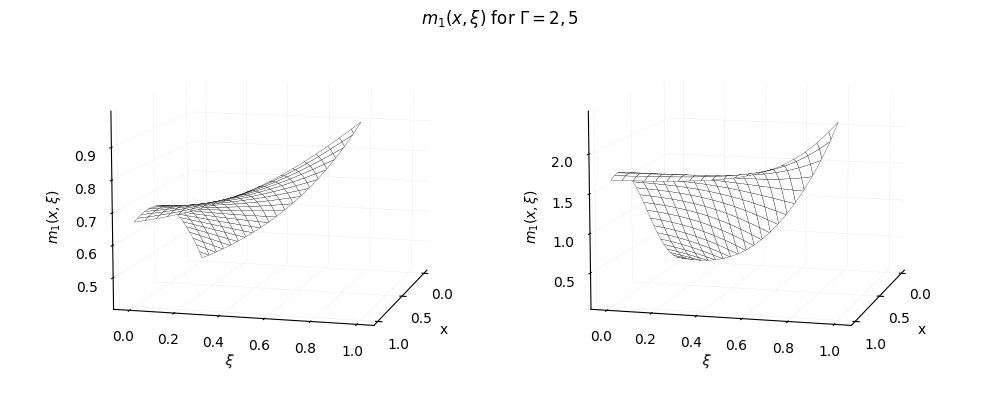}
\includegraphics[width=0.405\textwidth]{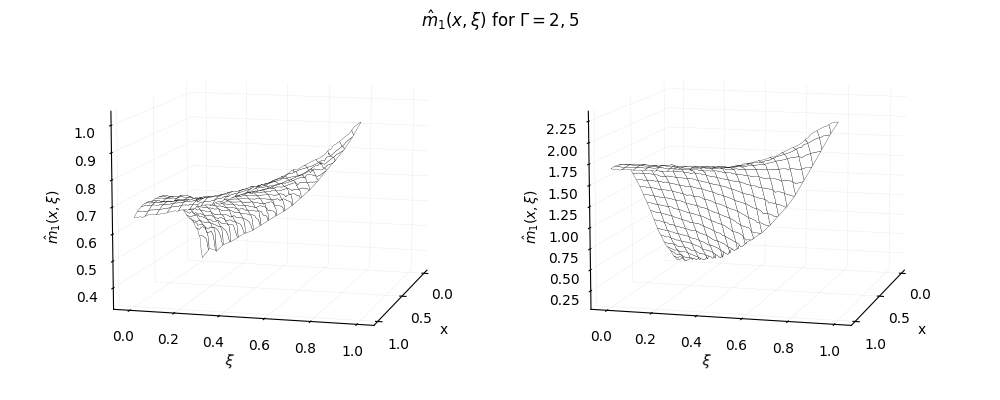}
\includegraphics[width=0.405\textwidth]{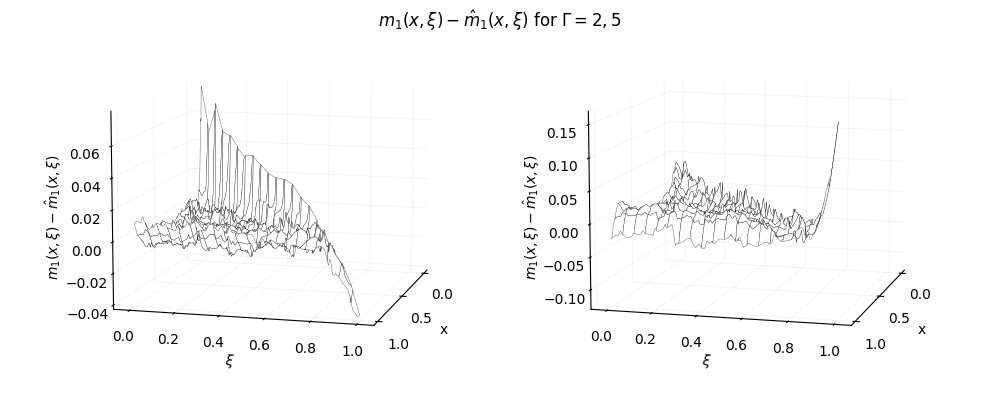}
\vspace{-0.15cm}
\caption{The kernel of $m_1(x,\xi)$, $\hat m_1(x,\xi)$ and $m_1(x,\xi)-\hat m_1(x,\xi)$.} \label{figure-m1}
\end{figure}
\begin{figure}[t]
\centering
\includegraphics[width=0.405\textwidth]{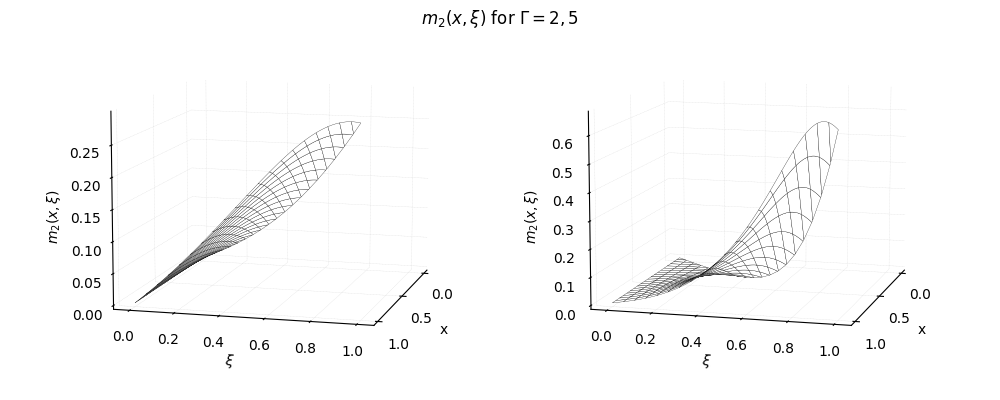}
\includegraphics[width=0.405\textwidth]{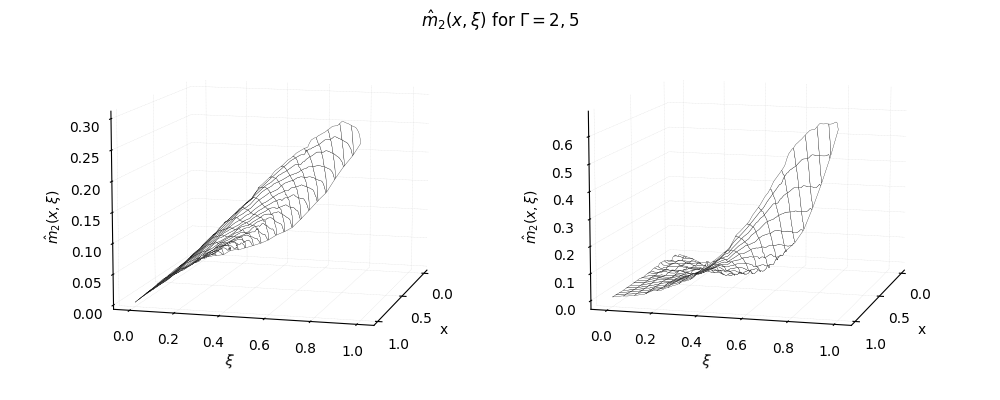}
\includegraphics[width=0.405\textwidth]{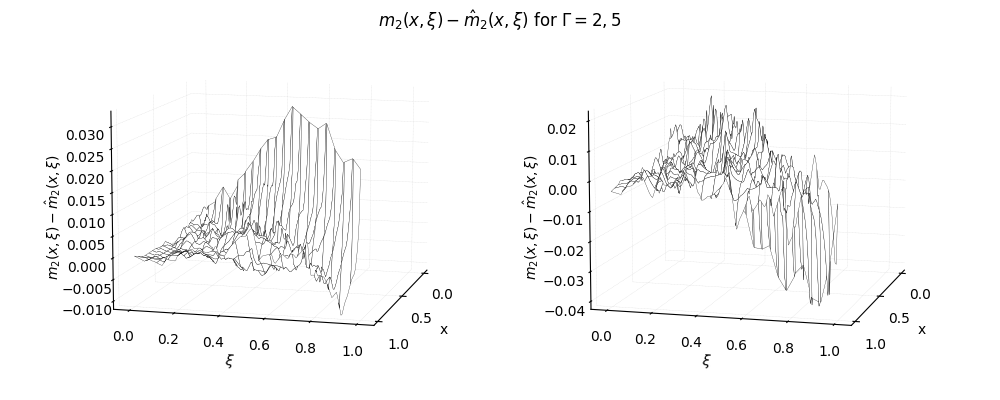}
\vspace{-0.15cm}
\caption{The kernel of $m_2(x,\xi)$, $\hat m_2(x,\xi)$ and $k_2(x,\xi)-\hat m_2(x,\xi)$.}  \label{figure-m2}
\end{figure}
\begin{figure}
\centering
\includegraphics[width=0.405\textwidth]{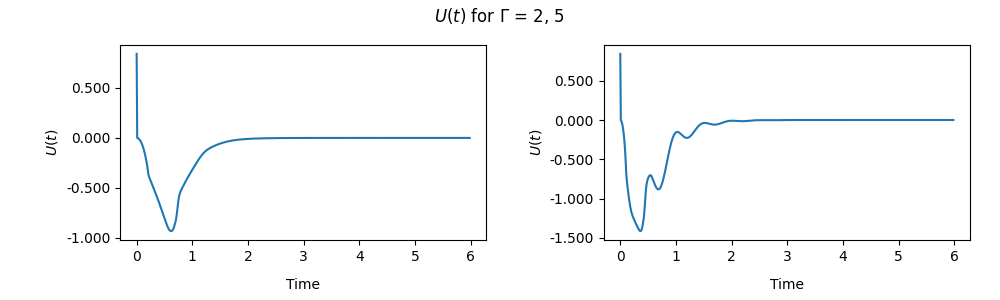}
\includegraphics[width=0.405\textwidth]{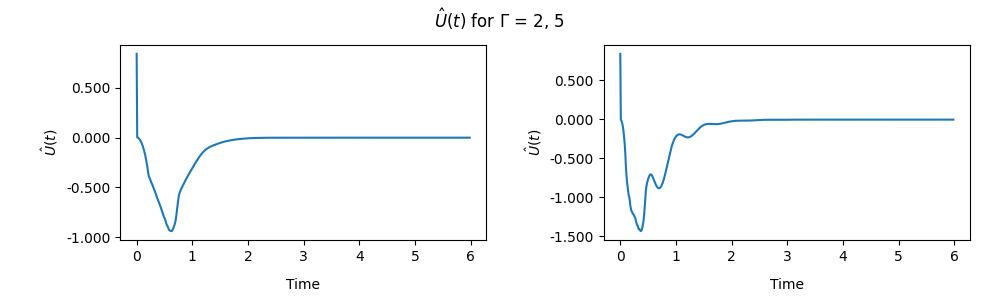}
\includegraphics[width=0.405\textwidth]{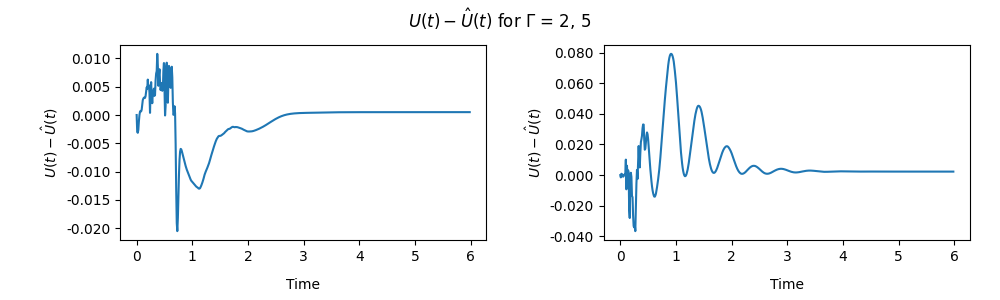}
\caption{Control law  $U$, $\hat U$ and $U-\hat U$.}\label{figure-U}
\end{figure}

\begin{figure}[h]
\centering
\includegraphics[width=0.21\textwidth]{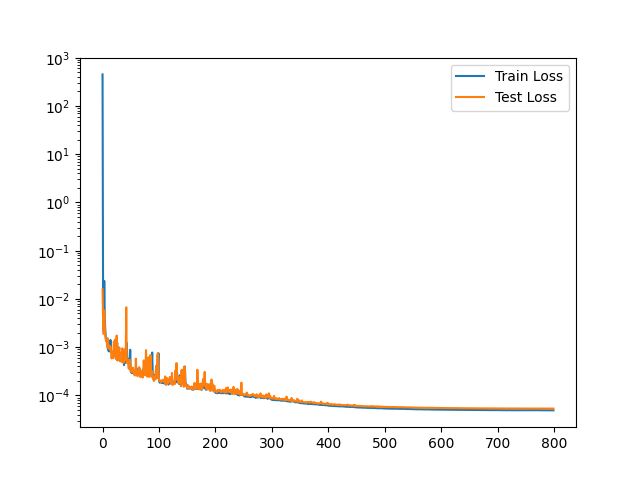}
\includegraphics[width=0.21\textwidth]{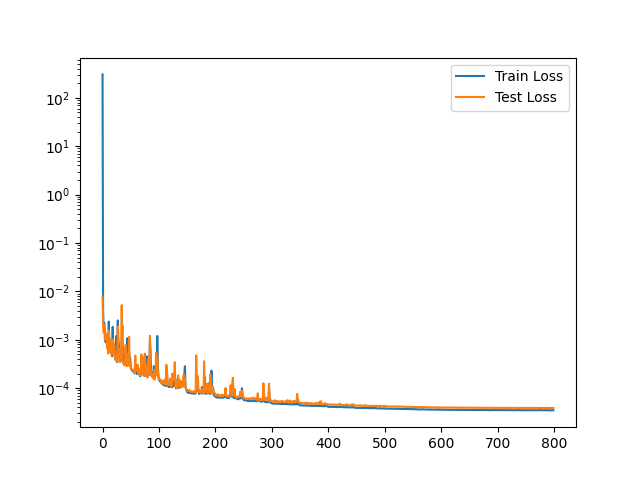}
\includegraphics[width=0.21\textwidth]{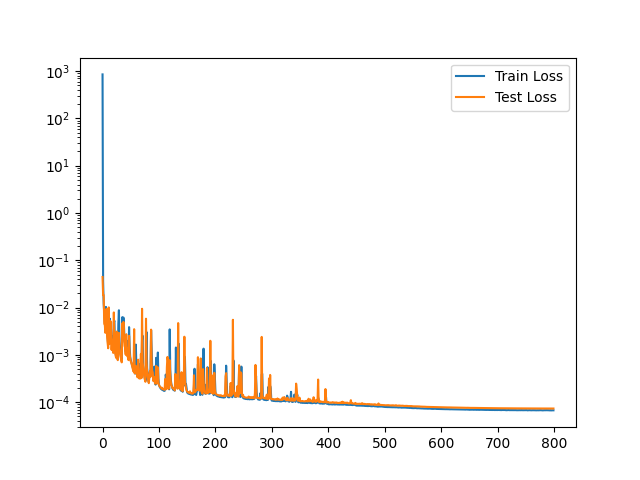}
\includegraphics[width=0.21\textwidth]{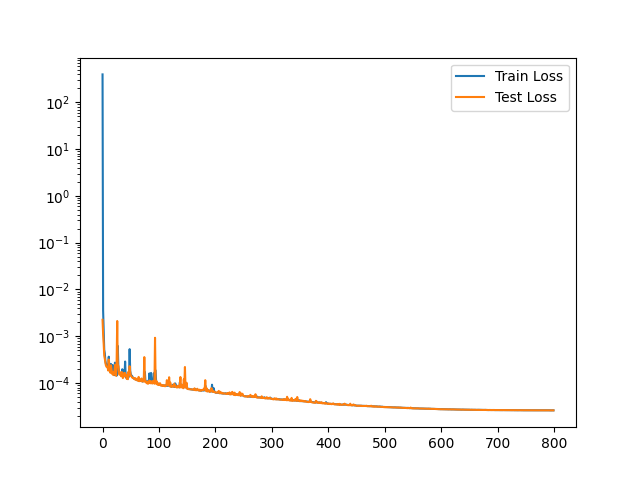}
\vspace{-0.15cm}
\caption{The train and test loss for $(\lambda(x),\mu(x),\omega(x),\sigma(x),$ $\theta(x),q)\to (k_1(x,\xi),\  k_2(x,\xi),\ m_1(x,\xi),\  m_2(x,\xi))$.} \label{Train-loss}
\end{figure}
\begin{figure}[h]
\centering
\includegraphics[width=0.21\textwidth]{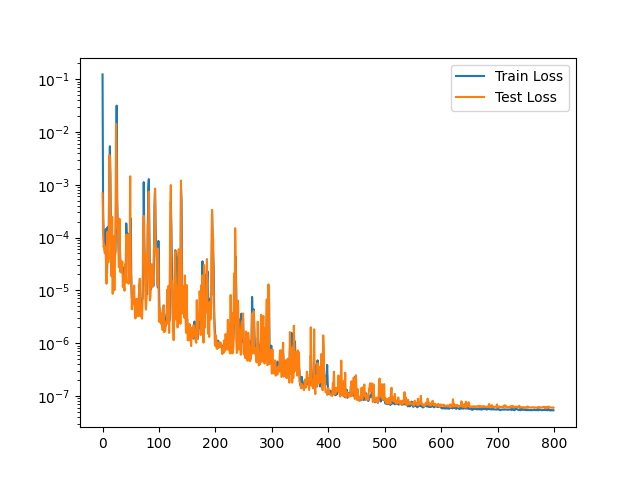}
\vspace{-0.15cm}
\caption{The train and test loss for $(\lambda(x),\mu(x),\omega(x),\sigma(x),$ $\theta(x),q)\to U$.} \label{Train-loss-control}
\end{figure}
\begin{figure}
\centering
\includegraphics[width=0.21\textwidth]{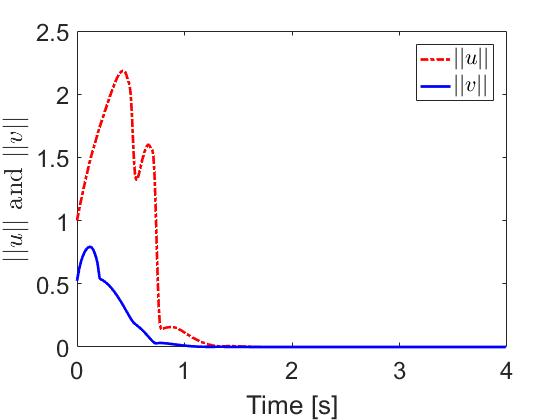}
\includegraphics[width=0.21\textwidth]{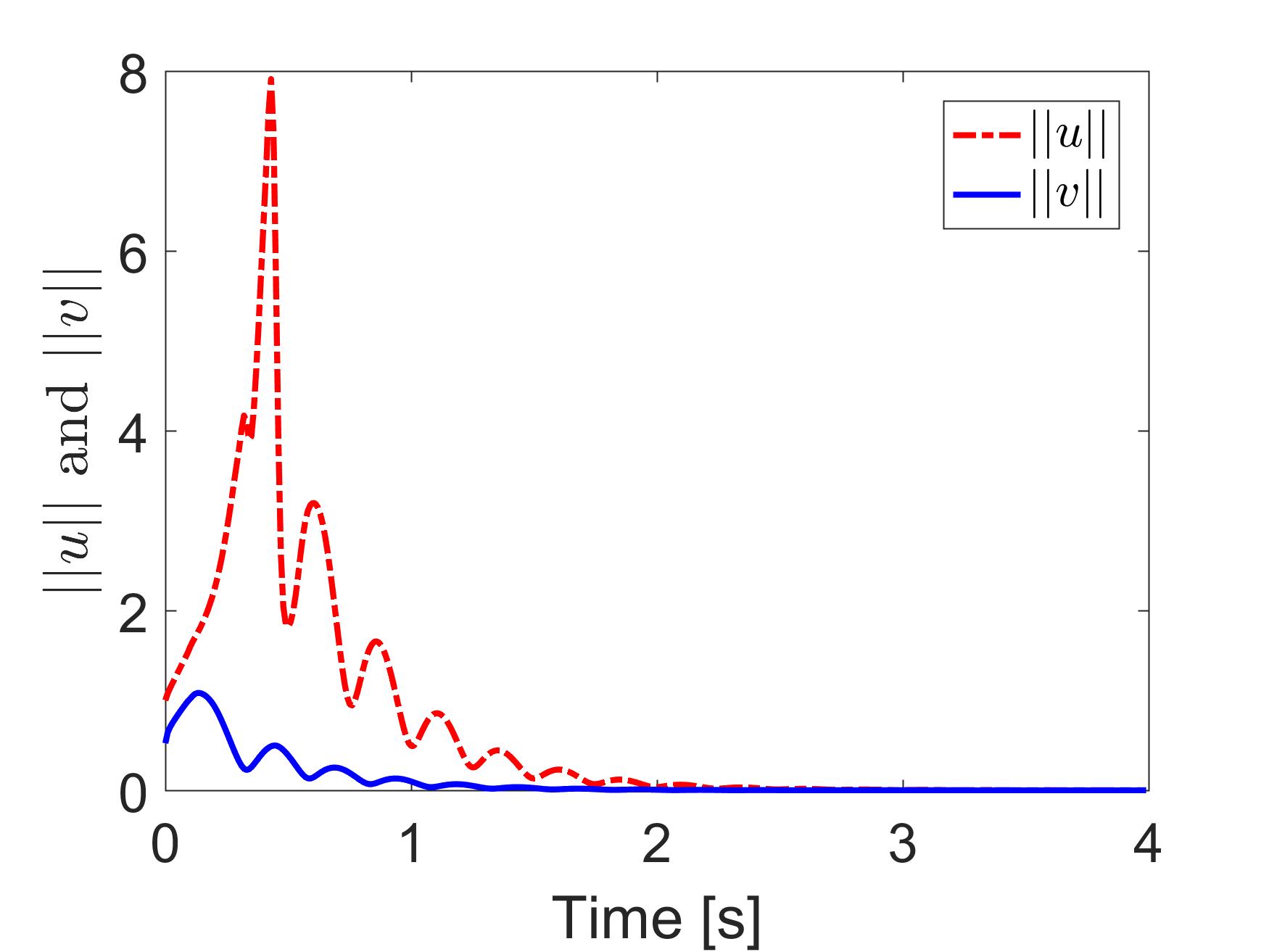}
\includegraphics[width=0.21\textwidth]{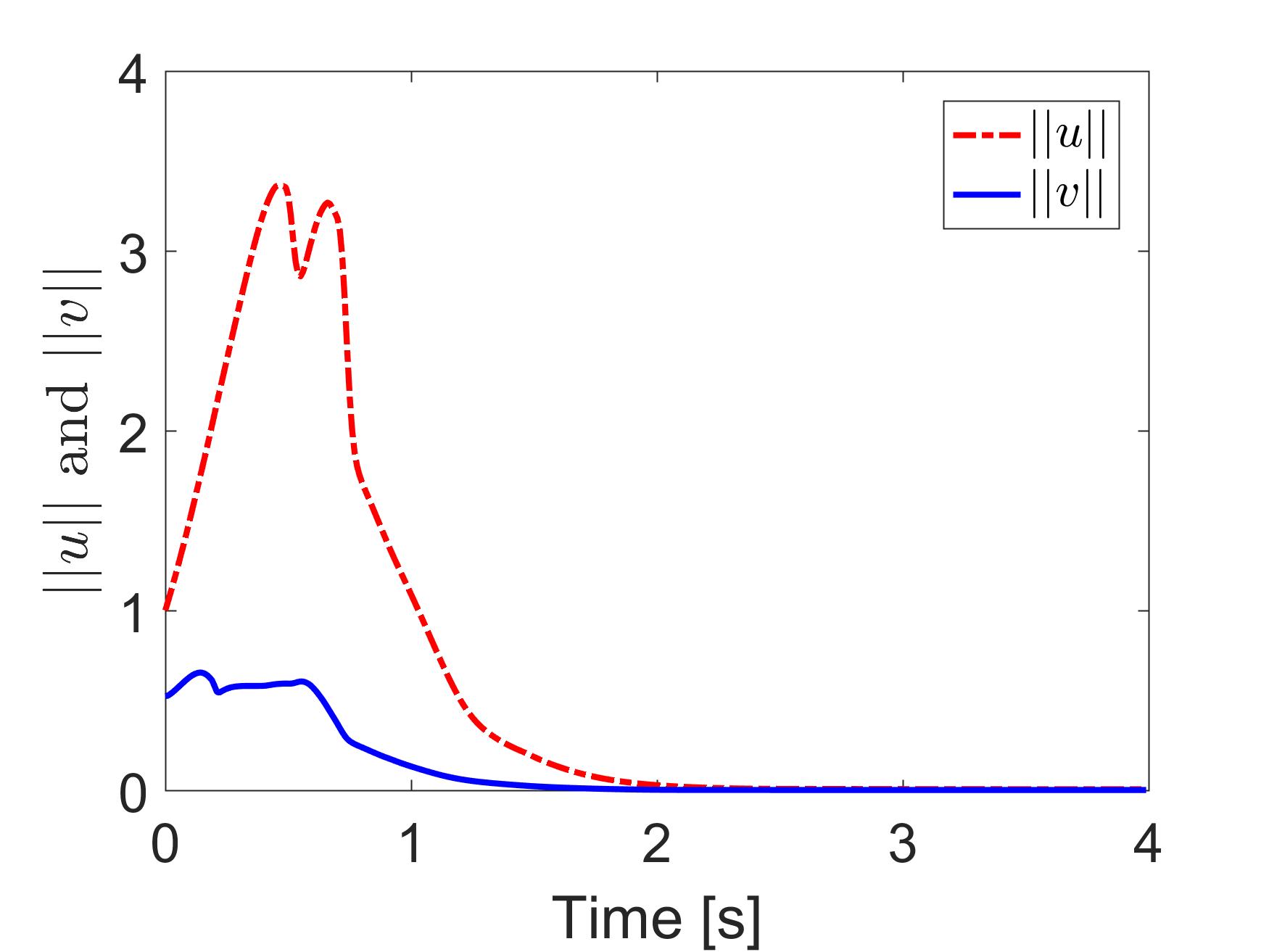}
\includegraphics[width=0.21\textwidth]{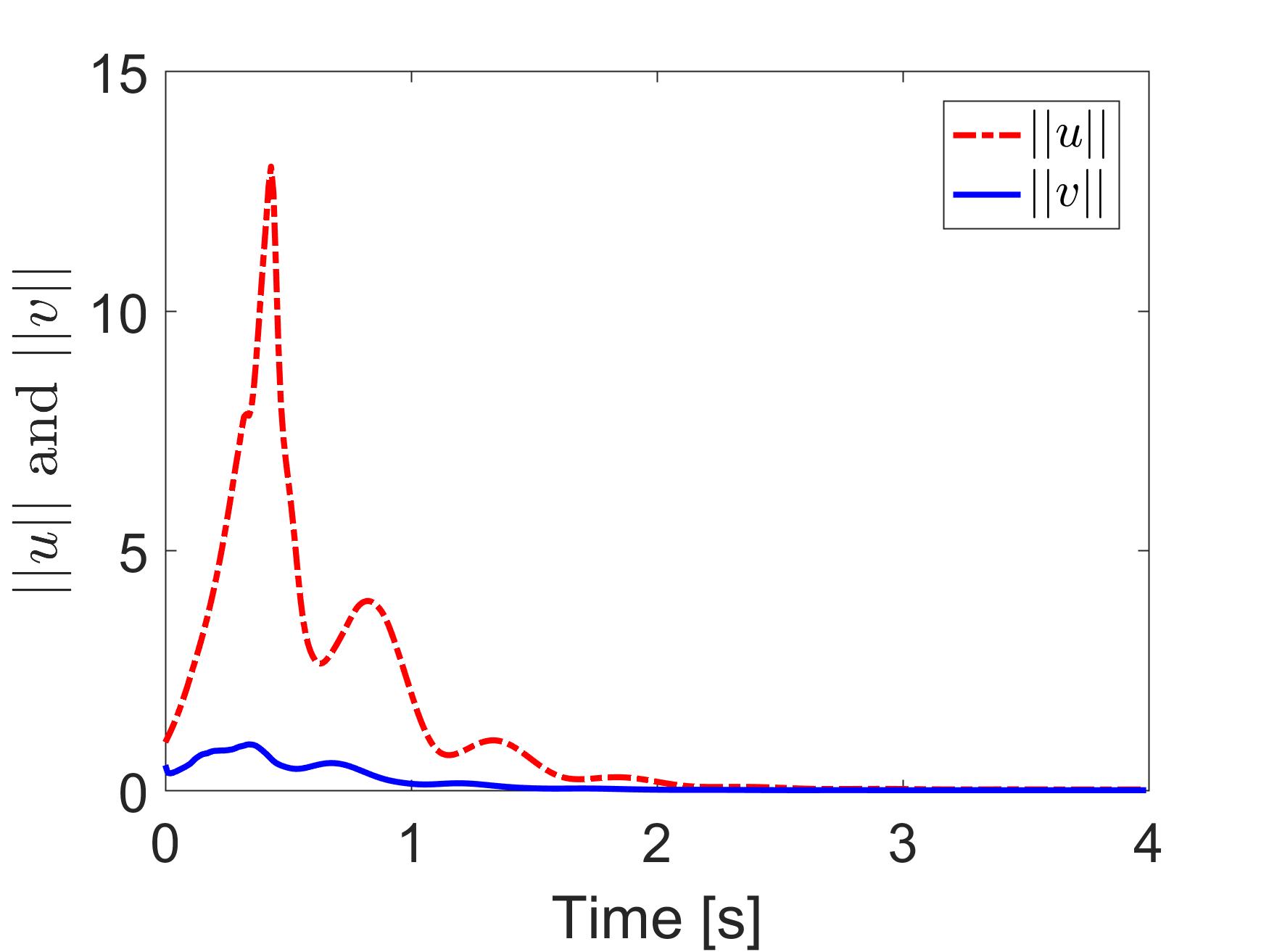}
\vspace{-0.15cm}
\caption{ \textbf{Upper row:}  closed-loop solutions with the observer kernels $m_1(x,\xi)$, $m_2(x,\xi)$, and the control law $U(t)$, respectively. \textbf{Lower row:}  closed-loop solutions with the observer kernels $\hat m_1(x,\xi)$, $\hat m_2(x,\xi)$, and control law $\hat U(t)$.}\label{figure-u}
\end{figure}
\begin{figure}
\centering
\includegraphics[width=0.21\textwidth]{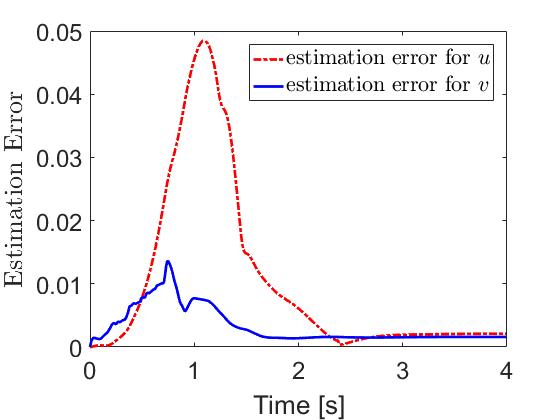}
\includegraphics[width=0.21\textwidth]{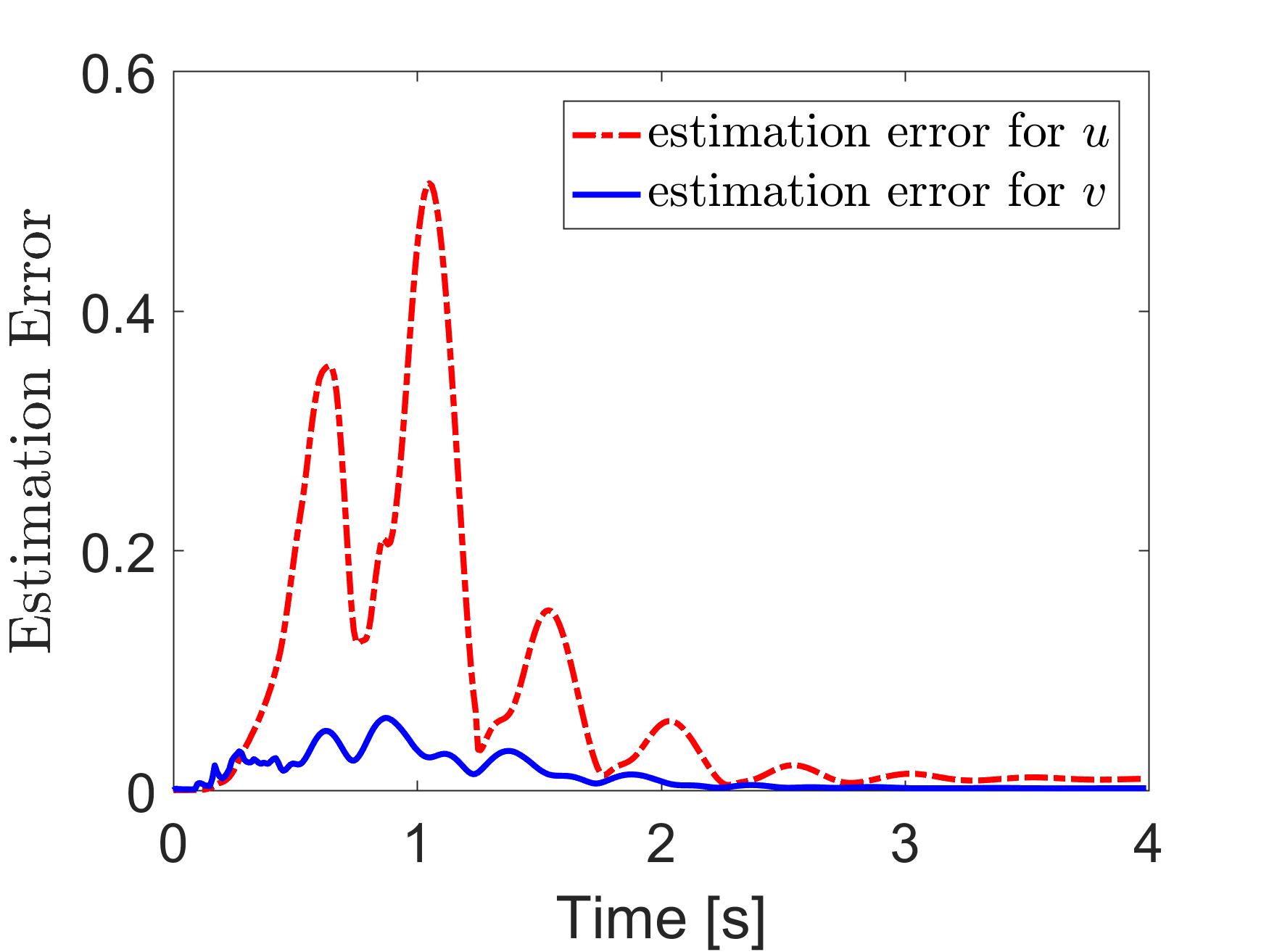}
\vspace{-0.15cm}
\caption{Residual error induced by the DeepONet approximation.}\label{figure-erroru}
\end{figure}
\begin{figure}
\centering
\includegraphics[width=0.21\textwidth]{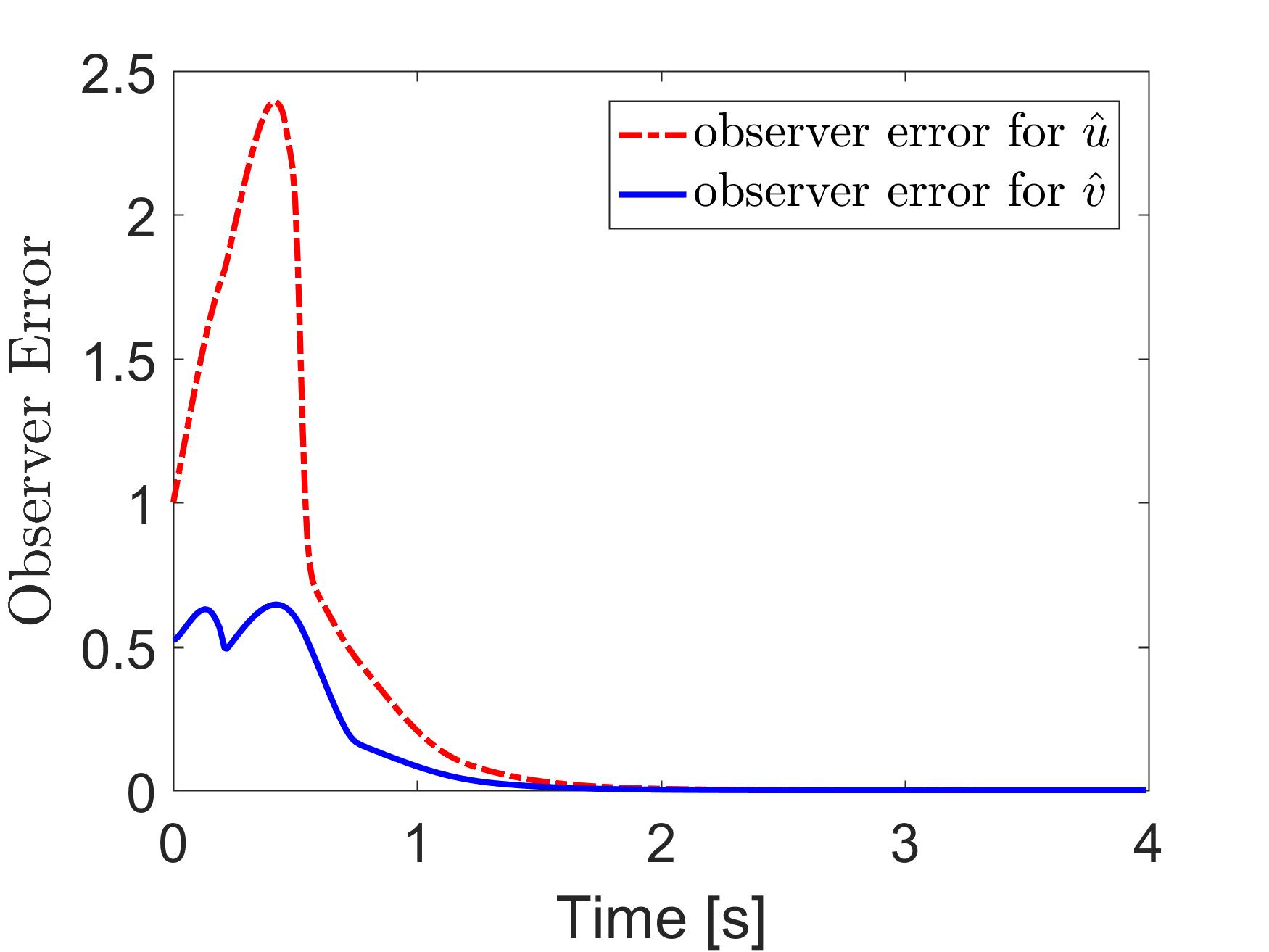}
\includegraphics[width=0.21\textwidth]{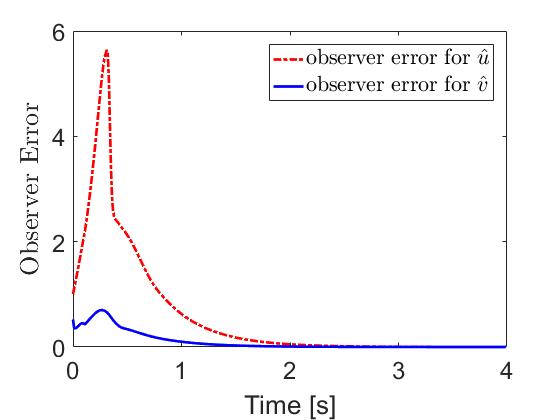}
\vspace{-0.15cm}
\caption{Observer error with observer kernels $\hat m_1(x,\xi)$ and $\hat m_2(x,\xi)$ and control law $\hat U(t)$.}\label{figure-observer-error}
\end{figure}

\section{Conclusion} \label{conclusion}
In this paper, we demonstrate the potential application of DeepONet to a PDE backstepping boundary control law of $2\times 2$  linear coupled hyperbolic PDE system. Generally, these systems follow distributed dynamics and arise in modeling systems like water canals, traffic, power transmission, and oil drilling systems.  In Section \ref{intro}, we  discussed major turning points in designing boundary controllers for broader classes of coupled hyperbolic  PDEs, advocating for DeepONet's use as a proof-based Machine Learning method for  PDE control.  PDE Backstepping-based DeepONet combines data-driven methods with deductive Lyapunov arguments using reliable data to expedite gain function computation from a known plant's model. Key results include \emph{Global Exponential Stability (GES)} for $L^2$ initial data with approximated kernel functions and \emph{Semi-global Practical Exponential Stability (SG-PES)} when the observer state is learned and fed to the controller.

\bibliographystyle{elsarticle-num}

\bibliography{refDataBase} 
%

%
%


       
\end{document}